\documentclass[11pt]{amsart}
\usepackage{mathrsfs}
\usepackage{}
\usepackage{xypic}
\usepackage{amsfonts}
\usepackage[pagewise]{lineno}

\usepackage{amssymb}
\usepackage{bbm}

\theoremstyle{plain}
\usepackage[colorlinks,
            linkcolor=black,
           anchorcolor=black,
            citecolor=black
            ]{hyperref}
\newtheorem{thm}{Theorem}[section]
\newtheorem{cor}[thm]{Corollary}
\newtheorem{lem}[thm]{Lemma}
\newtheorem{exa}[thm]{Example}
\newtheorem{prop}[thm]{Proposition}
\newtheorem{defi}[thm]{Definition}
\newtheorem{rem}[thm]{Remark}
\numberwithin{equation}{section}

\def\kk{\mathbbm{k}}

\begin{document}
\title{Bilinear forms on the Green rings of finite dimensional Hopf algebras}
\author{Zhihua Wang}
\address{Z. Wang\newline Department of Mathematics, Taizhou College, Nanjing Normal University,
Taizhou 225300, China}
\email{mailzhihua@126.com}
\author{Libin Li}
\address{L. Li\newline School of Mathematical Science, Yangzhou University,
Yangzhou 225002, China}
\email{lbli@yzu.edu.cn}
\author{Yinhuo Zhang}
\address{Y. Zhang\newline Department of Mathematics and Statistics, University of Hasselt, Universitaire Campus, 3590 Diepeenbeek, Belgium }
\email{yinhuo.zhang@uhasselt.be}
\date{}

\begin{abstract}  In this paper,  we study the Green ring and the stable Green ring of a finite dimensional Hopf algebra by means of bilinear forms. We show that the Green ring of a Hopf algebra of finite representation type is a Frobenius algebra over $\mathbb{Z}$ with a dual basis associated to almost split sequences.  On the stable Green ring we define a new bilinear form which is more accurate to determine the bi-Frobenius algebra structure on the stable Green ring. We show that the complexified stable Green algebra is a group-like algebra, and hence a bi-Frobenius algebra, if the bilinear form on the stable Green ring is non-degenerate.
\end{abstract}
\maketitle

\section{\bf Introduction }
In \cite{WLZ1, WLZ2} we studied the Green rings of  finite dimensional pointed rank one Hopf algebras of both nilpotent and non-nilpotent type respectively.  One of the interesting properties possessed by those Green rings is that the complexified stable Green algebras are group-like algebras, and consequently bi-Frobenius algebras, introduced and investigated by Doi and Takeuchi (cf. \cite{Doi1,Doi2, Doi3, DT}). The notion of a  bi-Frobenius algebra is a  natural generalization of a finite dimensional Hopf algebra,  and possesses  many properties that a finite dimensional Hopf algebra does.  However, to find more examples of  bi-Frobenius algebras, which are not Hopf algebras, is not easy at all.

In the context of this paper, the stable Green rings of finite dimensional Hopf algebras may provide interesting examples of group-like algebras and bi-Frobenius algebras in certain circumstances. Moreover, these bi-Frobenius algebras are themselves transitive fusion rings coming from non-semisimple stable categories. To do so, our principal technical tools are the bilinear forms on the Green rings introduced, e.g., in \cite{Be, NR, Wi2}. These bilinear forms are defined by means of dimensions of spaces of morphisms. One of the forms on the Green ring induces a from on the stable Green ring. The induced form on the stable Green ring is degenerate in general, we give some equivalent conditions for the induced form to be non-degenerate. If the form is non-degenerate and the Hopf algebra considered is of finite representation type, the stable Green ring becomes a transitive fusion ring although the stable category of the Hopf algebra is not necessary semisimple. In this case, the complexified stable Green algebra is a group-like algebra, and hence a bi-Frobenius algebra. To see a concrete example of such a stable Green algebra, we present explicitly the stable Green algebra of a Radford Hopf algebra from the polynomial point of view. The paper is organized as follows.

Let $H$ be a finite dimensional Hopf algebra and $H$-mod the category of finite dimensional (left) $H$-modules. In Section 2, we use quantum traces of morphisms of $H$-modules to characterize when the trivial module $\kk$ is a direct summand of the decomposition of the tensor product of any two indecomposable modules (see Theorem \ref{thm2.7}). Consequently, we answer the question raised by Cibils \cite[Remark 5.8]{Ci}. In particular, we apply the techniques from \cite{GMS, Y} to determine whether or not the trivial module $\kk$ appears in the decomposition of the particular tensor product $X\otimes X^*$ (resp. $X^*\otimes X$) for any indecomposable module $X$. Most results stated in this section are useful for next sections.

In Section 3, we follow the approach of \cite{Be} and impose three bilinear forms on the Green ring $r(H)$ of the Hopf algebra $H$. One of the forms is the bilinear form determined by $\langle [X],[Y]\rangle_1=\dim_{\kk}\text{Hom}_H(X,Y)$. Another is the form $\langle[X],[Y]\rangle_2=\dim_{\kk}\mathcal {P}(X,Y)$, where $\mathcal {P}(X,Y)$ stands for the space of morphisms from $X$ to $Y$ which factor through a projective module. The two forms are both non-degenerate and they are the same up to a unit. The third is the form $\langle[X],[Y]\rangle_3=\langle[X],[Y]\rangle_1-\langle[X],[Y]\rangle_2$. This form is degenerate since the ideal $\mathcal {P}$ of $r(H)$ generated by projective $H$-modules is contained in the radical of the form $\langle-,-\rangle_3$. If $H$ is of finite representation type, we prove that the radical of the form $\langle-,-\rangle_3$ is equal to $\mathcal {P}$ if and only if there are no periodic $H$-modules of even period.

In Section 4, we give several one-sided ideals of the Green ring $r(H)$ and use them to obtain key information about the nilpotent radical and central primitive idempotents of $r(H)$. The Green ring $r(H)$ is associated with an associative and non-degenerate bilinear form $([X],[Y]):=\langle[X],[Y^*]\rangle_1$. Thus, $r(H)$ is a Frobenius algebra over $\mathbb{Z}$ if $H$ is of finite representation type. In this case, the dual basis of $r(H)$ with respect to the form $(-,-)$ can be described partly by virtue of almost split sequences of $H$-modules. Let $\mathcal {P}^{\perp}$ be the subgroup of $r(H)$ which is orthogonal to $\mathcal {P}$ with respect to the form $(-,-)$. We show that the Grothendieck ring $G_0(H)$ of $H$ is a quotient ring of $r(H)$: $r(H)/\mathcal {P}^{\perp}\cong G_0(H)$. This isomorphism will be used in Section 5 to characterize when the nilpotent radical of $r(H)$ is equal to the intersection $\mathcal {P}\cap\mathcal {P}^{\perp}$.

In Section 5, we study the stable Green ring of $H$. The Green ring of the stable category $H$-\underline{mod} of $H$ is called the stable Green ring of $H$, denoted $r_{st}(H)$. As the stable category $H$-\underline{mod} is a quotient category of $H$-mod, the stable Green ring $r_{st}(H)$ is in fact a quotient ring of the Green ring $r(H)$: $r_{st}(H)\cong r(H)/\mathcal {P}$. This result can be used to define a new form $[-,-]_{st}$ on $r_{st}(H)$. The form $[-,-]_{st}$ is induced from the form $(-,-)$ on $r(H)$, and it is associative but degenerate in general. We determine both the left and right radicals of the form $[-,-]_{st}$ and give several equivalent conditions for the non-degeneracy of the form.
If $H$ is of finite representation type and the form $[-,-]_{st}$ is non-degenerate, the nilpotent radical of $r(H)$ is equal to $\mathcal {P}\cap\mathcal {P}^{\perp}$ (which is exactly the kernel of the Cartan map) if and only if the Grothendieck ring $G_0(H)$ is semiprime. The complexified stable Green algebra under the same assumption becomes a group-like algebra, and hence a bi-Frobenius algebra.

In Section 6, we consider the Radford Hopf algebra, a finite dimensional pointed Hopf algebra of rank one. The Green ring and the stable Green ring of the Radford Hopf algebra have been presented in \cite{WLZ2} by generators and relations. In this case, the bilinear form $[-,-]_{st}$ is non-degenerate, and hence the complexified stable Green algebra of the Radford Hopf algebra admits a bi-Frobenius algebra structure. We describe explicitly this structure in terms of polynomials.

Throughout, $H$ is an arbitrary finite dimensional Hopf algebra over an algebraically closed field $\mathbbm{k}$; all $H$-modules considered here are objects in $H$-mod. We denote $P_M$ and $I_M$ the projective cover and injective envelop of an $H$-module $M$ respectively. The tensor product $\otimes$ stands for $\otimes_{\mathbbm{k}}$. For any two $H$-modules $X$ and $Y$, the notation $X\mid Y$ (resp. $X\nmid Y$) means that $X$ is (resp. is not) a direct summand of $Y$. For the theory of Hopf algebras, we refer to \cite{Mon, Swe}.

\section{\bf Quantum traces of morphisms}

In the study of the Green ring $r(H)$ of a Hopf algebra $H$, one of difficult problems is to determine whether or not the trivial module $\kk$ appears in the tensor product $X\otimes Y$ of two indecomposable modules $X$ and $Y$.  This problem has already been solved in the case of group algebras by Benson and Carlson \cite[Theorem 2.1]{BC}, in the case of  involutory Hopf algebras in terms of splitting trace modules \cite{GMS}, and in the case of Hopf algebras with the square of antipode being inner \cite[Theorem 2.4]{Y}. Motivated by these works, in this section we shall make use of the notion of  quantum traces to solve the aforementioned problem  for  any  finite dimension Hopf algebra. In particular,  we will look in the special case $X\otimes X^*$ (or $X^*\otimes X$) for an indecomposable module $X$,  and give various characterizations of  $\kk\mid X\otimes X^*$ or not, which will be used in the next section.

Recall that the Hom-space $\textrm{Hom}_{\kk}(X,Y)$ is an $H$-module defined by $(hf)(x)=\sum h_1f(S(h_2)x),$ for $h\in H, x\in X$ and $f\in\textrm{Hom}_{\kk}(X,Y).$ In the special case where $Y$ is the trivial module $\kk$, then $X^*:=\textrm{Hom}_{\kk}(X,\kk)$ is an $H$-module given by $(hf)(x)=f(S(h)x)$, for $h\in H$, $x\in X$ and $f\in X^*$.
The \emph{evaluation} of $X$ is the morphism $\text{ev}_{X}:X^*\otimes X\rightarrow\kk$ given by $\text{ev}_X(f\otimes x)=f(x).$ The \emph{coevaluation} of $X$ is the morphism $\text{coev}_X:\kk\rightarrow X\otimes X^*$ defined by $\text{coev}_X(1)=\sum_ix_i\otimes x_i^*,$ where $\{x_i\}$ is a basis of $X$ and   $\{x_i^*\}$ is its dual basis in $X^*$.

The \emph{left quantum trace} of $\theta\in\text{Hom}_H(X,X^{**})$ is defined by the composition:
\begin{equation}\label{geq1}\text{Tr}^L_X(\theta):\kk\xrightarrow{\text{coev}_{X}}X\otimes X^*\xrightarrow{\theta\otimes id_{X^*}}X^{**}\otimes X^*\xrightarrow{\text{ev}_{X^*}}\kk.\end{equation} Similarly, the \emph{right quantum trace} of  a morphism $\theta\in\text{Hom}_H(X^{**},X)$ is defined by
\begin{equation}\label{Gequ1}\text{Tr}^R_X(\theta):\kk\xrightarrow{\text{coev}_{X^*}}X^*\otimes X^{**}\xrightarrow{id_{X^*}\otimes \theta}X^*\otimes X\xrightarrow{\text{ev}_X}\kk.\end{equation}
Since $\text{End}_H(\kk)\cong\kk$, both $\text{Tr}^L_X(\theta)$ and $\text{Tr}^R_X(\theta)$ are elements in $\kk$.

\begin{rem}
Applying the duality functor $*$ to (\ref{geq1}) and (\ref{Gequ1}) respectively, one obtains that $\text{Tr}^L_X(\theta)=\text{Tr}^R_{X^*}(\theta^*)$ and $\text{Tr}^R_X(\theta)=\text{Tr}^L_{X^*}(\theta^*)$, see \cite[Proposition 1.37.1]{etingof}.
\end{rem}

\begin{rem}
Let $P$ be a projective $H$-module.
\begin{enumerate}
\item If $H$ is not semisimple, then $\text{Tr}^L_P(\theta)=0$ for any $\theta\in\text{Hom}_H(P,P^{**})$. Otherwise, the morphism $\text{coev}_{P}$ is a split monomorphism by (\ref{geq1}). In this case, $\kk\mid P\otimes P^*$. It follows that $\kk$ is projective, and hence $H$ is semisimple, a contradiction. Similarly, if $H$ is not semisimple, then $\text{Tr}^R_P(\theta)=0$ for any $\theta\in\text{Hom}_H(P^{**},P)$.
\item If $H$ is involutory, i.e., $S^2=id_H$, then the map $\theta: P\rightarrow P^{**}$ given by $\theta(x)(f)=f(x)$ for $x\in P$ and $f\in P^*$ is an $H$-module isomorphism. In this case, Tr$^L_P(\theta)=\text{Tr}^R_P(\theta^{-1})=\dim_{\kk}P$. This implies that an involutory Hopf algebra over a field $\kk$ of characteristic 0 is semisimple (the converse is also true, see \cite{LR}). In case the characteristic of $\kk$ is $p>0$ and $H$ is not semisimple, then $p\mid\dim_{\kk}P$, giving a result of Lorenz \cite[Theorem 2.3 (b)]{Lo}.
\end{enumerate}
\end{rem}

We need  the following  two canonical isomorphisms  later on.
\begin{lem}\label{glem4} \cite[Lemma 2.1.6]{BA} For $X,Y,Z\in H$-mod, we have the following canonical isomorphisms functorial in $X$, $Y$ and $Z$:
\begin{enumerate}
  \item[(1)] $\Phi_{X,Y,Z}:\textrm{Hom}_H(X\otimes Y,Z)\rightarrow\textrm{Hom}_H(X,Z\otimes Y^*), \  \Phi_{X,Y,Z}(\alpha)=(\alpha\otimes id_{Y^*})\circ(id_X\otimes \text{coev}_Y)$.
  \item[(2)] $\Psi_{X,Y,Z}:\textrm{Hom}_H(X,Y\otimes Z)\rightarrow\textrm{Hom}_H(Y^*\otimes X,Z), \ \Psi_{X,Y,Z}(\gamma)=(\text{ev}_Y\otimes id_Z)\circ(id_{Y^*}\otimes\gamma)$.
\end{enumerate}
\end{lem}

Since the above two morphisms $\Phi_{X,Y,Z}$ and $\Psi_{X,Y,Z}$  will be used often throughout the paper, we give their inverse maps in detail.
$\Phi^{-1}_{X,Y,Z}(\beta)=(id_Z\otimes \text{ev}_Y)\circ(\beta\otimes id_Y)$ for $\beta\in\textrm{Hom}_H(X,Z\otimes Y^*)$, and  $\Psi^{-1}_{X,Y,Z}(\delta)=(id_Y\otimes\delta)\circ(\text{coev}_Y\otimes id_X)$ for $\delta\in\textrm{Hom}_H(Y^*\otimes X,Z)$. The two canonical isomorphisms satisfy the following properties.

\begin{prop}\label{gprop1} Let $X$ be an indecomposable $H$-module and $\theta: X\rightarrow X^{**}$ an $H$-module isomorphism. For any $Y\in H$-mod, we have the following:
\begin{enumerate}
  \item The canonical isomorphism $\text{Hom}_H(Y\otimes X^*,\kk)\xrightarrow{\Phi_{Y,X^*,\kk}}\text{Hom}_H(Y,X^{**})$ preserves split epimorphisms.
  \item The canonical isomorphism $\textrm{Hom}_H(Y,X)\xrightarrow{\Psi_{Y,X,\kk}}\textrm{Hom}_H(X^*\otimes Y,\kk)$
 reflects split epimorphisms.
\end{enumerate}
\end{prop}
\proof (1) If the map $\alpha\in\text{Hom}_H(Y\otimes X^*,\kk)$ is a split epimorphism, there is some $\beta\in\text{Hom}_H(\kk,Y\otimes X^*)$ such that $\alpha\circ\beta=id_{\kk}.$ For the map $\beta$, there is some $\gamma\in\text{Hom}_H(X,Y)$ such that $\beta=\Phi_{\kk,X,Y}(\gamma)$. Note that the composition $\Phi_{Y,X^*,\kk}(\alpha)\circ\gamma\circ\theta^{-1}\in\text{End}_H(X^{**})$. If $\Phi_{Y,X^*,\kk}(\alpha)\circ\gamma\circ\theta^{-1}\in\text{radEnd}_H(X^{**})$, then $(\Phi_{Y,X^*,\kk}(\alpha)\circ\gamma\circ\theta^{-1})\otimes id_{X^*}\in\text{radEnd}_H(X^{**}\otimes X^*)$. Hence the endomorphism of $\kk$:
$\text{ev}_{X^*}\circ((\Phi_{Y,X^*,\kk}(\alpha)\circ\gamma\circ\theta^{-1}\circ\theta)\otimes id_{X^*})\circ \text{coev}_X,$
factoring through $(\Phi_{Y,X^*,\kk}(\alpha)\circ\gamma\circ\theta^{-1})\otimes id_{X^*}$ , is zero. However,
\begin{linenomath}
\begin{align*}&\ \ \ \text{ev}_{X^*}\circ((\Phi_{Y,X^*,\kk}(\alpha)\circ\gamma\circ\theta^{-1}\circ\theta)\otimes id_{X^*})\circ \text{coev}_X\\
&=\Phi^{-1}_{Y,X^*,\kk}(\Phi_{Y,X^*,\kk}(\alpha))\circ\Phi_{\kk,X,Y}(\gamma)\\
&=\alpha\circ\beta\\
&=id_{\kk},
\end{align*}
\end{linenomath}
a contradiction. This implies that $\Phi_{Y,X^*,\kk}(\alpha)\circ\gamma\circ\theta^{-1}$ is an automorphism of $X^{**}$ since $\text{End}_H(X^{**})$ is local. Thus, the map $\Phi_{Y,X^*,\kk}(\alpha)$ is a split epimorphism.

(2)  The proof is similar. \qed

As an immediate consequence of Proposition \ref{gprop1}, we have the following.
\begin{cor}\label{gprop2}Let $X$ and $Y$ be two indecomposable $H$-modules and  assume $X\cong X^{**}$.
\begin{enumerate}
  \item[(1)] If $\kk\mid Y\otimes X^*$, then $Y\cong X^{**}$.
  \item[(2)] If $\kk\mid X^*\otimes Y$, then $X\cong Y$.
\end{enumerate}
\end{cor}

To pursue Corollary \ref{gprop2} even further, we need some preparations.
For any integer $m>0$, the $m$-th power of the duality functor $*$ on $X$ is denoted $X^{*m}$. If $\{x_i\}$ is a basis of $X$, we denote by $\{x_i^{*m}\}$ the basis of $X^{*m}$  dual to the basis $\{x_i^{*m-1}\}$ of $X^{*m-1}$, i.e., $\langle x_i^{*m},x_j^{*m-1}\rangle=\delta_{i,j}$. With these notations, we have the following.

\begin{lem}\label{glem1}Let $X$ be an indecomposable $H$-module.
\begin{enumerate}
  \item For any $\theta\in\text{Hom}_H(X,X^{**})$, if $\text{Tr}^L_X(\theta)\neq0$,  then $\theta$ is an isomorphism.
  \item For any $\theta\in\text{Hom}_H(X^{**},X)$, if $\text{Tr}^R_X(\theta)\neq0$,  then $\theta$ is an isomorphism.
\end{enumerate}
\end{lem}
\proof We only prove Part (1) and the proof of Part (2) is similar.
Denote by $\mathbf{A}$ the transformation matrix of $\theta\in\text{Hom}_H(X,X^{**})$ with respect to the bases $\{x_i\}$ and $\{x^{**}_i\}$. The left quantum trace of $\theta$ is $\text{Tr}^L_X(\theta)=\text{tr}(\mathbf{A})$, the usual trace of the matrix $\mathbf{A}$. Since $H$ is of finite dimension, the order of $S^2$ is finite by Radford's formula on $S^4$ and the Nichols-Z$\ddot{o}$ller Theorem. Suppose that $S^{2n}=id_H$. Then the map
\begin{linenomath}
$$\text{Id}:X^{*2n}\rightarrow X,\  \sum_i\lambda_ix_i^{*2n}\mapsto\sum_i\lambda_ix_i$$
\end{linenomath}
is an $H$-module isomorphism and the transformation matrix of the map $\text{Id}$ with respect to the basis $\{x_i^{*2n}\}$ of $X^{*2n}$ and the basis $\{x_i\}$ of $X$ is the identity matrix.
Consider the following composition map:
\begin{linenomath}
$$\Theta:X\xrightarrow{\theta}X^{**}\xrightarrow{\theta^{**}}X^{****}\rightarrow\cdots\rightarrow X^{*2n-2}\xrightarrow{\theta^{*2n-2}}X^{*2n}\xrightarrow{\text{Id}}X.$$
\end{linenomath}
Note that the matrix of the map $\Theta$ from $X$ to itself with respect to the basis $\{x_i\}$ of $X$ is $\mathbf{A}^n$. Since $\text{End}_H(X)$ is local, the map $\Theta$ is either nilpotent or isomorphic. If $\Theta$ is nilpotent, so is $\mathbf{A}^n$, and hence $\mathbf{A}$ is nilpotent. This implies that $\text{Tr}^L_X(\theta)=\text{tr}(\mathbf{A})=0$, a contradiction. Thus, $\Theta$ is an isomorphism, and so is the map $\theta$.
\qed

Cibils in \cite[Remark 5.8]{Ci}  raised the following question: when is the trivial module a direct summand of the tensor product of two indecomposable modules over a finite dimensional Hopf algebra with antipode of order bigger than 2?  We are now  ready to answer this question using  quantum traces.

\begin{thm}\label{thm2.7}
Let $X$ and $Y$ be two indecomposable $H$-modules.
\begin{enumerate}
  \item $\kk\mid Y\otimes X^*$ if and only if there are isomorphisms $f:X\rightarrow Y$ and $g:Y\rightarrow X^{**}$ such that
$\text{Tr}^L_X(g\circ f)\neq0$.
  \item $\kk\mid X^*\otimes Y$ if and only if there are isomorphisms $f:X^{**}\rightarrow Y$ and $g:Y\rightarrow X$ such that
$\text{Tr}^R_X(g\circ f)\neq0$.
\end{enumerate}
\end{thm}
\proof We only prove Part (1) and the same argument applies to Part (2). If $f:X\rightarrow Y$ and $g:Y\rightarrow X^{**}$ are two isomorphisms such that $\text{Tr}^L_X(g\circ f)\neq0$, then
\begin{linenomath}
$$0\neq\text{Tr}^L_X(g\circ f)=\text{ev}_{X^*}\circ(g\otimes id_{X^*})\circ(f\otimes id_{X^*})\circ \text{coev}_X.$$
\end{linenomath}
This implies that the map $(f\otimes id_{X^*})\circ \text{coev}_X:\kk\rightarrow Y\otimes X^*$ is a split monomorphism, and hence $\kk\mid Y\otimes X^*$. Conversely, if $\kk\mid Y\otimes X^*$, there are maps $\alpha: \kk\rightarrow Y\otimes X^*$ and $\beta:Y\otimes X^*\rightarrow\kk$ such that $\beta\circ\alpha=id_{\kk}$. For the map $\alpha$, by Lemma \ref{glem4}, there is a map $f:X\rightarrow Y$ such that
\begin{linenomath}
$$\alpha=\Phi_{\kk,X,Y}(f)=(f\otimes id_{X^*})\circ(id_{\kk}\otimes \text{coev}_X).$$
\end{linenomath}
For the map $\beta$, there is a map $g:Y\rightarrow X^{**}$ such that
\begin{linenomath}
$$\beta=\Phi^{-1}_{Y,X^*,\kk}(g)=(id_{\kk}\otimes \text{ev}_{X^*})\circ(g\otimes id_{X^*}).$$
\end{linenomath} Thus, we have:
\begin{linenomath}
\begin{align*}
\text{Tr}^L_X(g\circ f)&=\text{ev}_{X^*}\circ(g\otimes id_{X^*})\circ(f\otimes id_{X^*})\circ \text{coev}_X\\
&=(id_{\kk}\otimes \text{ev}_{X^*})\circ(g\otimes id_{X^*})\circ(f\otimes id_{X^*})\circ(id_{\kk}\otimes \text{coev}_X)\\
&=\beta\circ\alpha\\
&=id_{\kk}.
\end{align*}
\end{linenomath}
The composition $g\circ f$ is an isomorphism following from Lemma \ref{glem1}. Thus, $f$ and $g$ are both isomorphisms.\qed

Given two objects  $X,Y\in H$-mod, one knows little in general  about how to decompose the tensor product $X\otimes Y$ into a direct sum of indecomposable modules. However, there are still some rules that the decomposition should follow as shown in the following.

\begin{prop}\label{gg13}
Let $X,Y,M\in H$-mod with $X$ and $M$ being indecomposable.
\begin{enumerate}
  \item[(1)] If $\kk\mid M\otimes M^*$ and $M\mid X\otimes Y$, then $\kk\mid X\otimes X^*$ and $X\mid M\otimes Y^*$.
  \item[(2)] If $\kk\mid M^*\otimes M$ and $M\mid Y\otimes X$, then $\kk\mid X^*\otimes X$ and $X\mid Y^*\otimes M$.
\end{enumerate}
\end{prop}
\proof (1) We only prove Part (1) and the proof of Part (2) is similar.  The conditions $\kk\mid M\otimes M^*$ and $M\mid X\otimes Y$ imply that $\kk\mid X\otimes Y\otimes M^*$. Suppose $Y\otimes M^*\cong\bigoplus_iN_i^*$ for some indecomposable modules $N_i$. Then there is an indecomposable module $N_i$ such that $\kk\mid X\otimes N_i^*$. By Theorem \ref{thm2.7} (1), we obtain $X\cong N_i\cong N_i^{**}$. It follows that $\kk\mid X\otimes N_i^*\cong X\otimes X^*$. Note that $\kk\mid M\otimes M^*$ implies that $M\cong M^{**}$. Then $X\cong N_i^{**}$ implies that $X\mid (Y\otimes M^*)^*\cong M\otimes Y^*$, as desired.
\qed

In the rest of this section, $H$ will be a non-semisimple Hopf algebra. We shall take another approach to characterize when the trivial module $\kk$ appears in the decomposition of  the tensor product $X^*\otimes X$ (resp. $X\otimes X^*$) for an indecomposable module $X$. For the special case where the square of the antipode is inner, we refer to \cite{GMS, Y}. Suppose
\begin{equation}\label{equ2}0\rightarrow\tau(\kk)\rightarrow E\xrightarrow{\sigma}\kk\rightarrow0\end{equation}
is an almost split sequence ending at the trivial module $\kk$. Tensoring (over $\kk$) the sequence (\ref{equ2})
with an indecomposable module $X$, we obtain the following two short exact sequences:
\begin{equation}\label{equ3}0\rightarrow\tau(\kk)\otimes X\rightarrow E\otimes X\xrightarrow{\sigma\otimes id_X}X\rightarrow0,\end{equation}
\begin{equation}\label{equ4}0\rightarrow X\otimes\tau(\kk)\rightarrow X\otimes E\xrightarrow{id_X\otimes\sigma}X\rightarrow0.\end{equation}

We need the following lemma, its proof is straightforward if one applies Lemma  \ref{glem4}.
\begin{lem} For $X,Y\in H$-mod, the following diagrams are commutative:
\begin{equation}\label{gdiag2}
\xymatrix{
 \textrm{Hom}_H(Y,X\otimes E)\ar[d]_{\Psi_{Y,X,E}}\ar[r]^{(id_X\otimes\sigma)_*}&\textrm{Hom}_H(Y,X)\ar[d]^{\Psi_{Y,X,\kk}}&\\
 \textrm{Hom}_H(X^*\otimes Y,E) \ar[r]^{\sigma_*} &\textrm{Hom}_H(X^*\otimes Y,\kk),&}
\end{equation}
\quad
 \begin{equation}\label{gdiag22}
\xymatrix{
 \textrm{Hom}_H(Y\otimes X,E)\ar[d]_{\Phi_{Y,X,E}} \ar[r]^{\sigma_*}&\textrm{Hom}_H(Y\otimes X,\kk)\ar[d]^{\Phi_{Y,X,\kk}}& \\
 \textrm{Hom}_H(Y,E\otimes X^{*}) \ar[r]^{(\sigma\otimes id_{X^{*}})_*}&\textrm{Hom}_H(Y,X^{*}).&}
\end{equation}
\end{lem}

\begin{prop}\label{gg10}
Let $X$ be an indecomposable $H$-module. The following are equivalent:
\begin{enumerate}
  \item[(1)] $\kk\nmid X^*\otimes X$
  \item[(2)] The map $\textrm{Hom}_H(X^*\otimes X,E)\xrightarrow{\sigma_*}\textrm{Hom}_H(X^*\otimes X,\kk)$ is surjective.
  \item[(3)] The map $\textrm{Hom}_H(X,X\otimes E)\xrightarrow{(id_X\otimes\sigma)_*}\textrm{Hom}_H(X,X)$ is surjective.
  \item[(4)] The map $X\otimes E\xrightarrow{id_X\otimes\sigma}X$ is a split epimorphism.
  \item[(5)] The map $E\otimes X^*\xrightarrow{\sigma\otimes id_{X^*}}X^*$ is a split epimorphism.
\end{enumerate}
\end{prop}
\proof $(1)\Leftrightarrow(2)$. If $\kk\nmid X^*\otimes X$,  then for any $\alpha\in\textrm{Hom}_H(X^*\otimes X,\kk)$, the map $\alpha$ is not a split epimorphism.  Since $\sigma$ is right almost split from $E$ to $\kk$, there is a map $\beta$ from $X^*\otimes X$ to $E$ such that $\sigma\circ\beta=\alpha$. This implies that $\sigma_*$ is surjective.
Conversely, if the map $\sigma_*$ is surjective, then $\kk\nmid X^*\otimes X$. Otherwise, by Theorem \ref{thm2.7} (2), there is an isomorphism $\theta:X^{**}\rightarrow X$ such that $\text{Tr}^R_X(\theta)=id_{\kk}$.
For the map $\text{ev}_X:X^*\otimes X\rightarrow\kk$, there is some $\beta\in\textrm{Hom}_H(X^*\otimes X,E)$ such that $\sigma\circ\beta=\text{ev}_{X}$ since the map $\sigma_*$ is surjective. It follows that $id_{\kk}=\text{Tr}_X^R(\theta)=\text{ev}_{X}\circ(id_{X^*}\otimes\theta)\circ\text{coev}_{X^*}
=\sigma\circ\beta\circ(id_{X^*}\otimes\theta)\circ\text{coev}_{X^*}.$ We obtain that the map $\sigma$ is a split epimorphism, a contradiction to the fact that $\sigma$ is right almost split.

$(2)\Leftrightarrow(3)$. According to the commutative diagram (\ref{gdiag2}), we have the following commutative diagram:
\begin{linenomath}
\begin{equation*}
\xymatrix{
 \textrm{Hom}_H(X,X\otimes E)\ar[d]_{\Psi_{X,X,E}} \ar[r]^{(id_X\otimes\sigma)_*}& \textrm{Hom}_H(X,X)\ar[d]^{\Psi_{X,X,\kk}}& \\
 \textrm{Hom}_H(X^*\otimes X,E) \ar[r]^{\sigma_*}& \textrm{Hom}_H(X^*\otimes X,\kk).&}
\end{equation*}
\end{linenomath}
It follows  that $\sigma_*$ is surjective if and only if $(id_X\otimes\sigma)_*$ is surjective.

$(3)\Leftrightarrow(4)$. If $(id_X\otimes\sigma)_*$ is surjective, for the identity map $id_X$, there is a map $\alpha\in\text{Hom}_H(X,X\otimes E)$ such that $(id_X\otimes\sigma)_*(\alpha)=id_X$. Then $(id_X\otimes\sigma)\circ\alpha=id_X$, and hence $id_X\otimes\sigma$ is a split epimorphism. Conversely, if $id_X\otimes\sigma$ is a split epimorphism, there is $\alpha\in\text{Hom}_H(X,X\otimes E)$ such that $(id_X\otimes\sigma)\circ\alpha=id_X$. For any $\beta\in \text{Hom}_H(X,X)$, we have $(id_X\otimes\sigma)_*(\alpha\circ\beta)=\beta$. It yields that the map $(id_X\otimes\sigma)_*$ is surjective.

$(2)\Leftrightarrow(5)$. It follows from the commutative diagram (\ref{gdiag22}) that the diagram
\begin{linenomath}
\begin{equation*}
\xymatrix{
  \textrm{Hom}_H(X^*\otimes X,E) \ar[d]_{\Phi_{X^*,X,E}} \ar[r]^{\sigma_*}& \textrm{Hom}_H(X^*\otimes X,\kk)\ar[d]^{\Phi_{X^*,X,\kk}}& \\
 \textrm{Hom}_H(X^*,E\otimes X^*) \ar[r]^{(\sigma\otimes id_{X^*})_*}&\textrm{Hom}_H(X^*,X^*)&}
\end{equation*}
\end{linenomath}
is commutative. Thus,  $\sigma_*$ is surjective if and only if $(\sigma\otimes id_{X^*})_*$ is surjective.
If $(\sigma\otimes id_{X^*})_*$ is surjective, for the identity map $id_{X^*}$, there is $\alpha\in\textrm{Hom}_H(X^*,E\otimes X^*)$ such that $id_{X^*}=(\sigma\otimes id_{X^*})_*(\alpha)=(\sigma\otimes id_{X^*})\circ\alpha$. This implies that $\sigma\otimes id_{X^*}$ is a split epimorphism. Conversely, if $\sigma\otimes id_{X^*}$ is a split epimorphism, there is $\alpha\in\textrm{Hom}_H(X^*,E\otimes X^*)$ such that $(\sigma\otimes id_{X^*})\circ\alpha=id_{X^*}$. For any $\beta\in\textrm{Hom}_H(X^*,X^*)$, we obtain that $(\sigma\otimes id_{X^*})_*(\alpha\circ\beta)=\beta$. It follows that the map $(\sigma\otimes id_{X^*})_*$ is surjective.
\qed

Similarly, there are some equivalent conditions for $\kk\nmid X\otimes X^*$. However, we only need the following characterization, which is  useful in the study of the Green ring of $H$.
\begin{prop}\label{gg07}
Let $X$ be an indecomposable $H$-module. The following are equivalent:
\begin{enumerate}
  \item[(1)] $\kk\nmid X\otimes X^*$
  \item[(2)] The map $E\otimes X\xrightarrow{\sigma\otimes id_X}X$ is a split epimorphism.
\end{enumerate}
\end{prop}
\proof Let $Y$ be indecomposable such that $Y^*\cong X$ (such a $Y$ exists as the order of $S^2$ is finite). Then $\kk\nmid X\otimes X^*$ if and only if $\kk\nmid(Y^{*}\otimes Y)^*$ if and only if $\kk\nmid Y^{*}\otimes Y$. By Proposition \ref{gg10}, this is precisely $\sigma\otimes id_{Y^*}$ is a split epimorphism, as desired. \qed

Although we have characterized $\kk\nmid X^*\otimes X$ and $\kk\nmid X\otimes X^*$ respectively in the previous two propositions, we still find the following characterizations of $\kk\mid X^*\otimes X$ and  $\kk\mid X\otimes X^*$ useful.
\begin{prop}\label{gg11}
Let $X$ be an indecomposable $H$-module. The following are equivalent:
\begin{enumerate}
  \item[(1)] $\kk\mid X^*\otimes X$.
  \item[(2)] The map $X\otimes E\xrightarrow{id_X\otimes\sigma}X$ is right almost split.
\end{enumerate}
\end{prop}
\proof If $id_X\otimes\sigma$ is right almost split, it is not a split epimorphism. By Proposition \ref{gg10},
we have $\kk\mid X^*\otimes X$. Conversely, if $\kk\mid X^*\otimes X$, by Proposition \ref{gg10}, the map $id_X\otimes\sigma$ is not a split epimorphism. The condition $\kk\mid X^*\otimes X$ implies that $X\cong X^{**}$.
For any $\alpha\in\textrm{Hom}_H(Y,X)$ which is not split epimorphism, the map $\Psi_{Y,X,\kk}(\alpha)\in\textrm{Hom}_H(X^*\otimes Y,\kk)$ is also not split epimorphism by Proposition \ref{gprop1} (2). For the map $\Psi_{Y,X,\kk}(\alpha)$, there is a map $\beta\in\textrm{Hom}_H(X^*\otimes Y,E)$ such that
\begin{linenomath}
$$\sigma\circ\beta=\Psi_{Y,X,\kk}(\alpha)$$
\end{linenomath} since $\sigma$ is right almost split. Note that $\Psi^{-1}_{Y,X,E}(\beta)\in\textrm{Hom}_H(Y,X\otimes E)$. We claim that the map $\Psi^{-1}_{Y,X,E}(\beta)$ satisfies the relation $(id_X\otimes\sigma)\circ\Psi^{-1}_{Y,X,E}(\beta)=\alpha$, and hence $id_X\otimes\sigma$ is right almost split. In fact, the commutative diagram (\ref{gdiag2}) implies  that
\begin{linenomath}
$$\Psi_{Y,X,\kk}\circ(id_X\otimes\sigma)_*=\sigma_*\circ\Psi_{Y,X,E}.$$
\end{linenomath} Therefore, we have:
\begin{linenomath}
\begin{align*}
\alpha&=\Psi^{-1}_{Y,X,\kk}(\sigma\circ\beta)\\
&=(\Psi^{-1}_{Y,X,\kk}\circ \sigma_*)(\beta)\\
&=((id_X\otimes\sigma)_*\circ\Psi_{Y,X,E}^{-1})(\beta)\\
&=(id_X\otimes\sigma)\circ\Psi^{-1}_{Y,X,E}(\beta),
\end{align*}
\end{linenomath}
which completes the proof. \qed

Similarly, we have the following.

\begin{prop}\label{gprop8}
Let $X$ be an indecomposable $H$-module. The following are equivalent:
\begin{enumerate}
  \item[(1)] $\kk\mid X\otimes X^*$.
  \item[(2)] The map $E\otimes X\xrightarrow{\sigma\otimes id_X}X$ is right almost split.
\end{enumerate}
\end{prop}

\begin{rem}
An indecomposable module satisfying one of the equivalent conditions in Proposition \ref{gg11}  or  in Proposition \ref{gprop8} is called a \textit{splitting trace module}, see e.g., \cite{EGST, GMS, Y}.
\end{rem}

\section{\bf Bilinear forms on Green rings}
As shown in \cite{Be}, an approach to study the Green ring
of a finite group is through bilinear forms defined by dimensions of morphism spaces. In this section, we follow the same approach and define similar bilinear forms on the Green ring $r(H)$ of $H$. As we shall see, these bilinear forms can be used to investigate some properties of $r(H)$ presented in the next section.

Let $F(H)$ be the free abelian group generated by isomorphism classes $[X]$ of $X\in H$-mod. The group $F(H)$ is in fact a ring with a multiplication given by the tensor product $[X][Y]=[X\otimes Y]$. The \textit{Green ring} (or the \textit{representation ring}) $r(H)$ of  $H$ is defined to be the quotient ring of $F(H)$ modulo the relations $[X\oplus Y]=[X]+[Y]$, for $X,Y\in H$-mod. The identity of the associative ring  $r(H)$ is represented by the trivial module $[\mathbbm{k}]$. The set ind$(H)$ consisting of all isomorphism classes of indecomposable objects in $H$-mod forms a $\mathbb{Z}$-basis of $r(H)$, see e.g., \cite{Ch, DH, HOYZ, LH, Wa}.

The \textit{Grothendieck ring} $G_0(H)$ of $H$ is the quotient ring of $F(H)$ modulo all short exact sequences of $H$-modules, i.e., $[Y]=[X]+[Z]$ if $0\rightarrow X\rightarrow Y\rightarrow Z\rightarrow0$ is exact. The Grothendieck ring $G_0(H)$ possesses a $\mathbb{Z}$-basis given by isomorphism classes of simple $H$-modules. Both $r(H)$ and  $G_0(H)$ are  augmented $\mathbb{Z}$-algebras with the dimension augmentation. There is a natural ring epimorphism from $r(H)$ to $G_0(H)$ given by
\begin{equation}\label{equgel}
\varphi:r(H)\rightarrow G_0(H),\ [M]\mapsto\sum_{[V]}[M:V][V],
\end{equation} where $[M:V]$ is the multiplicity of $V$ in the composition sequence of $M$ and the sum $\sum_{[V]}$ runs over all non-isomorphic simple $H$-modules. If $H$ is semisimple, the map $\varphi$ is the identity map.


Let $Z$ be an  indecomposable $H$-module. If $Z$ is  non-projective, there is a unique almost split sequence $0\rightarrow X\rightarrow Y\rightarrow Z\rightarrow0$ ending at $Z$. We follow the notation given in \cite[Section 4, ChVI]{ARS} and denote by $\delta_{[Z]}$ the element $[X]-[Y]+[Z]$ in $r(H)$. In case $Z$ is  projective, we define $\delta_{[Z]}:=[Z]-[\textrm{rad}Z].$
The following is a weaker condition for $\delta_{[Z]}=[X]-[Y]+[Z]$ in $r(H)$.

\begin{prop}\label{gg14}
Let $Z$ be an indecomposable non-projective module, and $0\rightarrow X\rightarrow Y\xrightarrow{\alpha}Z\rightarrow0$ a short exact sequence. If the map $\alpha$ is right almost split, then $\delta_{[Z]}=[X]-[Y]+[Z].$
\end{prop}
\proof Since the sequence \begin{equation}\label{equ9}0\rightarrow X\rightarrow Y\xrightarrow{\alpha}Z\rightarrow0\end{equation} is exact and the map $\alpha$ is right almost split, it follows from \cite[Theorem 2.2, ChI]{ARS} that the middle term $Y$ has a decomposition $Y=Y_1\oplus Y_2$ such that the restriction of $\alpha$ to the summand $Y_1$, denoted  $\alpha|_{Y_1}$,  is right minimal, and the restriction of $\alpha$ to the summand $Y_2$ is zero. We obtain that $\alpha|_{Y_1}$ is right minimal as well as right almost split. By
\cite[Proposition 1.12, ChV]{ARS}, the sequence
$0\rightarrow \ker(\alpha|_{Y_1})\xrightarrow{\iota}Y_1\xrightarrow{\alpha|_{Y_1}}Z\rightarrow0$ is almost split, where $\iota$ is the inclusion map. Thus, $\delta_{[Z]}=[\ker(\alpha|_{Y_1})]-[Y_1]+[Z].$ Meanwhile, it is easy to see that the sequence
\begin{equation}\label{equ10}0\rightarrow \ker(\alpha|_{Y_1})\oplus Y_2\xrightarrow{\iota\coprod id_{Y_2}}Y_1\oplus Y_2\xrightarrow{\alpha}Z\rightarrow0\end{equation} is exact. Applying the short five lemma to the sequences (\ref{equ9}) and (\ref{equ10}), we obtain that $X\cong\ker(\alpha|_{Y_1})\oplus Y_2$. In this case,
\begin{linenomath}
\begin{align*}
\delta_{[Z]}&=[\ker(\alpha|_{Y_1})]-[Y_1]+[Z]\\
&=[\ker(\alpha|_{Y_1})\oplus Y_2]-[Y_1\oplus Y_2]+[Z]\\
&=[X]-[Y]+[Z],
\end{align*}
\end{linenomath}  as desired.
\qed

For any two objects $X,Y\in H$-mod, following \cite{Be, NR, Wi2} we define
\begin{equation*}
\langle[X],[Y]\rangle_1:=\dim_{\kk} \textrm{Hom}_H(X,Y).
\end{equation*}
Then, $\langle-,-\rangle_1$ extends to a $\mathbb{Z}$-bilinear form on $r(H)$.
The following results can be found from Proposition 4.1, Theorem 4.3 and  Theorem 4.4 in \cite[ChVI]{ARS}.

\begin{lem}\label{gg15}
The following hold in $r(H)$:
\begin{enumerate}
\item[$(1)$] For any two indecomposable modules $X$ and $Z$, $\langle[X],\delta_{[Z]}\rangle_1=\delta_{[X],[Z]}$, where $\delta_{[X],[Z]}$ is equal to 1 if $X\cong Z$, and 0 otherwise.
\item[$(2)$] For any $x\in r(H)$, $x=\sum_{[M]\in\textrm{ind}(H)}\langle x,\delta_{[M]}\rangle_1 [M]$.
\item[$(3)$] $\{\delta_{[M]}\mid [M]\in\textrm{ind}(H)\}$ is linearly independent in $r(H)$.
\item[$(4)$] $H$ is of finite representation type if and only if $\{\delta_{[M]}\mid [M]\in\textrm{ind}(H)\}$ forms a $\mathbb{Z}$-basis of $r(H)$.
\item[$(5)$] $H$ is of finite representation type if and only if $\{\delta_{[M]}\mid [M]\in\textrm{ind}(H)$ and $M$ not projective$\}$ forms a $\mathbb{Z}$-basis of $\ker\varphi$, where $\varphi$ is the map given in (\ref{equgel}).
\end{enumerate}
\end{lem}

\begin{rem}\label{rem2016} It follows from Lemma \ref{gg15} (2) that the form $\langle-,-\rangle_1$ is non-degenerate in the sense that given $0\neq x\in r(H)$, there is $y\in r(H)$ such that $\langle x,y\rangle_1\neq0$. If $H$ is of finite representation type, it can be seen from Lemma \ref{gg15} that the set $\{[M],\delta_{[M]}\mid [M]\in\text{ind}(H)\}$ forms a pair of dual bases of $r(H)$ with respect to the form $\langle-,-\rangle_1$. In this case, any $x$ in $r(H)$ can be written as follows: $x=\sum_{[M]\in\textrm{ind}(H)}\langle[M],x\rangle_1\delta_{[M]}.$
\end{rem}

We use the non-degeneracy of the form $\langle-,-\rangle_1$ to give an equivalent condition for $H$ to be of finite representation type.
\begin{prop}
The Hopf algebra $H$ is of finite representation type if and only if for any indecomposable module $X$, there are only finitely many indecomposable modules $M$ such that $\text{Hom}_H(M,X)\neq0$.
\end{prop}
\proof For any indecomposable module $X$, if there are only finitely many indecomposable modules $M$ such that $\text{Hom}_H(M,X)\neq0$, then the sum $\sum_{[M]\in\textrm{ind}(H)}\dim_{\kk}\text{Hom}_H(M,X)\delta_{[M]}$ is a finite sum. We have the following:
\begin{linenomath}
\begin{align*}&\ \ \ \ \langle[M],[X]-\sum_{[M]\in\textrm{ind}(H)}\dim_{\kk}\text{Hom}_H(M,X)\delta_{[M]}\rangle_1\\
&=\langle[M],[X]\rangle_1-\dim_{\kk}\text{Hom}_H(M,X)\\
&=0.
\end{align*}
\end{linenomath}
This implies that $[X]=\sum_{[M]\in\textrm{ind}(H)}\dim_{\kk}\text{Hom}_H(M,X)\delta_{[M]}$ by the non-degeneracy of the form $\langle-,-\rangle_1$. Thus, $\{\delta_{[M]}\mid [M]\in\textrm{ind}(H)\}$ is a basis of $r(H)$, and hence $H$ is of finite representation type by Lemma \ref{gg15} (4). \qed

Let $\mathcal {P}(X,Y)$ be the space of morphisms from $X$ to $Y$ which factor through a projective module. By a similar way to \cite{Be}, we define another bilinear form on $r(H)$ as follows:
\begin{linenomath}
\begin{equation*}
\langle[X],[Y]\rangle_2:=\dim_{\kk} \mathcal {P}(X,Y).
\end{equation*}
\end{linenomath}

Let $\ast$ denote the duality operator of $r(H)$ induced by the duality functor: $[X]^*=[X^*]$. Then $\ast$ is an anti-automorphism of $r(H)$. Obviously, if $S^2$ of $H$ is inner, then $\ast$ is an involution \cite{Lo}. The forms $\langle-,-\rangle_1$ and $\langle-,-\rangle_2$ both have the following properties.

\begin{prop}\label{prop22.2}Let $X$, $Y$ and $Z$ be $H$-modules.
\begin{enumerate}
  \item $\langle[X][Y],[Z]\rangle_1=\langle[X],[Z][Y]^*\rangle_1$ and $\langle[X],[Y][Z]\rangle_1=\langle[Y]^*[X],[Z]\rangle_1$.
  \item $\langle[X][Y],[Z]\rangle_2=\langle[X],[Z][Y]^*\rangle_2$ and $\langle[X],[Y][Z]\rangle_2=\langle[Y]^*[X],[Z]\rangle_2$.
\end{enumerate}
\end{prop}
\proof (1) It follows from Lemma \ref{glem4}.

(2) If $\alpha\in\text{Hom}_H(X\otimes Y,Z)$ factors through a projective module $P$, then $\Phi_{X,Y,Z}(\alpha)$ factors through the projective module $P\otimes Y^*$ by Lemma \ref{glem4} (1).
Thus, $\Phi_{X,Y,Z}(\mathcal {P}(X\otimes Y,Z))\subseteq\mathcal {P}(X,Z\otimes Y^*).$ Conversely, for any $\beta\in\mathcal {P}(X,Z\otimes Y^*)$ which factors through a projective module $P$, by Lemma \ref{glem4} (1), the map $\Phi^{-1}_{X,Y,Z}(\beta)$ factors through the projective module $P\otimes Y$. We obtain that $\Phi_{X,Y,Z}(\mathcal {P}(X\otimes Y,Z))=\mathcal {P}(X,Z\otimes Y^*).$ Similarly, $\Psi_{X,Y,Z}(\mathcal {P}(X,Y\otimes Z))=\mathcal {P}(Y^*\otimes X,Z).$ We are done.
\qed

Let $\Omega$ and $\Omega^{-1}$ denote the syzygy functor and cosyzygy functor of $H$-mod respectively. Namely, $\Omega M$ is the kernel of the projective cover $P_M\rightarrow M$, and $\Omega^{-1}M$ is the cokernel of the injective envelope $M\rightarrow I_M$. Denote by $\delta_{[M]}^*$ the image of $\delta_{[M]}$ under the duality operator $\ast$ of $r(H)$.
The following is a generalization of \cite[Proposition 2.1]{Be} to the case of the Green ring $r(H)$.
We omit the proof since it is similar to the proof of \cite[Proposition 2.1]{Be}.

\begin{lem}\label{prop1515}Let $M$ be an indecomposable $H$-module and $P_{\kk}$ the projective cover of the trivial module $\kk$. The following hold in $r(H)$:
\begin{enumerate}
  \item $([I_M]-[\Omega^{-1}M])\delta_{[P_{\kk}]}=\delta_{[P_{\kk}]}([I_M]-[\Omega^{-1}M])=[M]$ and $([P_M]-[\Omega M])\delta^*_{[P_{\kk}]}=\delta^*_{[P_{\kk}]}([P_M]-[\Omega M])=[M]$. Moreover, $\delta_{[P_{\kk}]}\delta^*_{[P_{\kk}]}=\delta^*_{[P_{\kk}]}\delta_{[P_{\kk}]}=1$.
  \item $[M]\delta_{[P_{\kk}]}=\delta_{[P_{\kk}]}[M]=[P_M]-[\Omega M]$ and $[M]\delta^*_{[P_{\kk}]}=\delta^*_{[P_{\kk}]}[M]=[I_M]-[\Omega^{-1}M]$. Thus, $\delta_{[P_{\kk}]}$ and $\delta^*_{[P_{\kk}]}$ are both central units of $r(H)$.
\end{enumerate}
\end{lem}

The following explores a relation between the forms $\langle-,-\rangle_1$ and $\langle-,-\rangle_2$. We refer to \cite[Corollary 2.3]{Be} for a similar result.

\begin{prop}\label{prop2.6} Let $X$ and $Y$ be two $H$-modules.
\begin{enumerate}
  \item $\langle[X],[Y]\rangle_2$ is equal to the multiplicity of $P_{\kk}$ in a direct sum decomposition of $Y^*\otimes X$.
  \item $\langle[X],[Y]\rangle_2=\langle[X],[Y]\delta_{[P_{\kk}]}\rangle_1=\langle[X]\delta^*_{[P_{\kk}]},[Y]\rangle_1$.
  \item $\langle[X],[Y]\rangle_1=\langle[X]\delta_{[P_{\kk}]},[Y]\rangle_2=\langle[X],[Y]\delta^*_{[P_{\kk}]}\rangle_2$.
\end{enumerate}
\end{prop}
\proof (1) For any non-zero morphism $\alpha\in\mathcal {P}(Y^*\otimes X,\kk)$, if $\alpha$ factors through an indecomposable projective module $P$, then $\alpha=\beta\circ\gamma$ for some $\beta:P\rightarrow\kk$ and $\gamma: Y^*\otimes X\rightarrow P$. Since $\beta$ is surjective, $P$ is the projective cover of $\kk$ and hence $P\cong P_{\kk}$. Note that rad$P_{\kk}$ is the unique maximal submodule of $P_{\kk}$. The image of the morphism $\gamma$ is either contained in rad$P_{\kk}$ or equal to $P_{\kk}$. For the former case, $\alpha=\beta\circ\gamma=0$, a contradiction. Thus, the morphism $\gamma$ is surjective, and hence $P_{\kk}$ is a direct summand of $Y^*\otimes X$. Now, if $\alpha$ factors through a projective module $P$ and $P\cong\bigoplus_i P_i$ for some indecomposable projective modules $P_i$. Then $\alpha=\sum_i\beta_i\circ\gamma_i$ for some $\beta_i:P_i\rightarrow\kk$ and $\gamma_i: Y^*\otimes X\rightarrow P_i$. We have proved that $\beta_i\circ\gamma_i\neq0$ if and only if $P_i\cong P_{\kk}$. It follows that $\dim_{\kk}\mathcal {P}(Y^*\otimes X,\kk)$ is equal to the multiplicity of $P_{\kk}$ in a direct sum decomposition of $Y^*\otimes X$, while the former is equal to $\dim_{\kk}\mathcal {P}(X,Y)$ by Proposition \ref{prop22.2} (2).

(2) It follows from Part (1) that $\langle[X],[Y]\rangle_2=\langle[Y]^*[X],\delta_{[P_{\kk}]}\rangle_1$. By Proposition \ref{prop22.2}, we have \begin{linenomath}
\begin{align*}\langle[Y]^*[X],\delta_{[P_{\kk}]}\rangle_1&=\langle[X],[Y]\delta_{[P_{\kk}]}\rangle_1\\
&=\langle[X]\delta^*_{[P_{\kk}]}\delta_{[P_{\kk}]},[Y]\delta_{[P_{\kk}]}\rangle_1\\
&=\langle[X]\delta^*_{[P_{\kk}]},[Y]\delta_{[P_{\kk}]}\delta^*_{[P_{\kk}]}\rangle_1\\&=\langle[X]\delta^*_{[P_{\kk}]},[Y]\rangle_1.
\end{align*}
\end{linenomath}

(3) It follows from Part (2) and the fact that $\delta_{[P_{\kk}]}\delta^*_{[P_{\kk}]}=\delta^*_{[P_{\kk}]}\delta_{[P_{\kk}]}=1$.
\qed

\begin{cor}\label{cor2015}
Let $X$ be an indecomposable $H$-module and $V$ a simple $H$-module. Then $\langle[X],[V]\rangle_2=\delta_{[X],[P_V]}$.
\end{cor}
\proof It follows from Proposition \ref{prop2.6} that $\langle[X],[V]\rangle_2=\langle[X],[V]\delta_{[P_{\kk}]}\rangle_1=\langle[X],[P_V]\rangle_1-\langle[X],[\Omega V]\rangle_1=1$ if $X\cong P_V$, and 0 otherwise.\qed

\begin{rem}
Let $H$ be of finite representation type. It follows from Proposition \ref{prop2.6} (3) that the set $\{[M]\delta_{[P_{\kk}]},\delta_{[M]}\mid [M]\in\text{ind}(H)\}$ or $\{[M],\delta_{[M]}\delta^*_{[P_{\kk}]}\mid [M]\in\text{ind}(H)\}$ forms a pair of dual bases of $r(H)$ with respect to the form $\langle-,-\rangle_2$. Hence the form $\langle-,-\rangle_2$ is the same as $\langle-,-\rangle_1$ up to a unit. Namely, $\langle-,-\rangle_1=\langle-\delta_{[P_{\kk}]},-\rangle_2=\langle-,-\delta^*_{[P_{\kk}]}\rangle_2$.
\end{rem}

For any two $H$-modules $X$ and $Y$, we define \begin{linenomath}$$\langle[X],[Y]\rangle_3:=\langle[X],[Y]\rangle_1-\langle[X],[Y]\rangle_2.$$\end{linenomath} It follows from Proposition \ref{prop2.6} that \begin{linenomath}$$\langle[X],[Y]\rangle_3=\langle[X],[Y](1-\delta_{[P_{\kk}]})\rangle_1=\langle[X](1-\delta^*_{[P_{\kk}]}),[Y]\rangle_1.$$\end{linenomath} Moreover, we have the following result.

\begin{prop}\label{prop2016}
Let $X$ and $Y$ be two $H$-modules.
\begin{enumerate}
  \item If $X$ is indecomposable and projective, then $\langle[X],[Y]\rangle_3=0$.
  \item If $X$ is indecomposable and non-projective, then \begin{linenomath}$$\langle[X],[Y]\rangle_3=\langle[X],[Y]\rangle_1+\langle[\Omega^{-1}X],[Y]\rangle_1-\sum_{[V]}[Y:V]\langle[\Omega^{-1}X],[V]\rangle_1,$$\end{linenomath} where the sum $\sum_{[V]}$ runs over all non-isomorphic simple $H$-modules and $[Y:V]$ is the multiplicity of $V$ in the composition sequence of $Y$. In particular, $\langle[X],[Y]\rangle_3=\langle[X],[Y]\rangle_1$ if $Y$ is simple.
\end{enumerate}
\end{prop}
\proof (1) It follows from the definition of the form $\langle-,-\rangle_3$.

(2) For any simple $H$-module $V$, on the one hand, $\langle[X],[V]\rangle_2=0$ by Corollary \ref{cor2015}. On the other hand, $\langle[X],[V]\rangle_2=\langle[X]\delta^*_{[P_{\kk}]},[V]\rangle_1=\langle[I_X]-[\Omega^{-1}X],[V]\rangle_1$. It follows that $\langle[I_X],[V]\rangle_1=\langle[\Omega^{-1}X],[V]\rangle_1$. Now
\begin{linenomath}
\begin{align*}
\langle[X],[Y]\rangle_3&=\langle[X](1-\delta^*_{[P_{\kk}]}),[Y]\rangle_1\\
&=\langle[X],[Y]\rangle_1+\langle[\Omega^{-1}X],[Y]\rangle_1-\langle[I_X],[Y]\rangle_1\\
&=\langle[X],[Y]\rangle_1+\langle[\Omega^{-1}X],[Y]\rangle_1-\sum_{[V]}[Y:V]\langle[I_X],[V]\rangle_1\\
&=\langle[X],[Y]\rangle_1+\langle[\Omega^{-1}X],[Y]\rangle_1-\sum_{[V]}[Y:V]\langle[\Omega^{-1}X],[V]\rangle_1,
\end{align*}
\end{linenomath} as desired.
\qed

The \textit{left radical} of the form $\langle-,-\rangle_3$ is the set $\{x\in r(H)\mid \langle x,y\rangle_3=0\ \text{for\ any}\ y\in r(H)\}$. It is equivalent to the set that $\{x\in r(H)\mid x(1-\delta^*_{[P_{\kk}]})=0\}$. Similarly, the \textit{right radical} of the form $\langle-,-\rangle_3$ is equivalent to the set $\{x\in r(H)\mid x(1-\delta_{[P_{\kk}]})=0\}$. The left and right radicals of the form coincide since $\delta_{[P_{\kk}]}\delta^*_{[P_{\kk}]}=\delta^*_{[P_{\kk}]}\delta_{[P_{\kk}]}=1$. Note that $[P](1-\delta_{[P_{\kk}]})=0$ for any projective module $P$. Thus, the ideal $\mathcal {P}$ of $r(H)$ generated by isomorphism classes of projective $H$-modules is contained in the radical of the form. For further results about the radical of the form, we need the following lemma.

\begin{lem}\label{g102}Let $M$ and $Z$ be two indecomposable $H$-modules.
\begin{enumerate}
  \item $\langle[M],\delta_{[Z]}\rangle_{3}
=\delta_{[M],[Z]}+\delta_{[\Omega^{-1}M],[Z]}-\delta_{[I_M],[Z]}.$
  \item $\delta_{[M]}\delta^*_{[P_{\kk}]}=\left\{
                                     \begin{array}{ll}
                                       -\delta_{[\Omega^{-1}M]}, & \hbox{M\ \text{is\ not\ projective};} \\
                                       \text{Soc} M, & \hbox{M\ \text{is\ projective}.}
                                     \end{array}
                                   \right.$
\end{enumerate}
\end{lem}
\proof (1) Note that $\langle[M],\delta_{[Z]}\rangle_{3}=\langle[M](1-\delta^*_{[P_{\kk}]}),\delta_{[Z]}\rangle_1$. It follows that
\begin{linenomath}\begin{align*}\langle[M](1-\delta^*_{[P_{\kk}]}),\delta_{[Z]}\rangle_1
&=\langle[M],\delta_{[Z]}\rangle_{1}-\langle[I_M]-[\Omega^{-1}M],\delta_{[Z]}\rangle_{1}\\
&=\langle[M],\delta_{[Z]}\rangle_{1}+\langle[\Omega^{-1}M],\delta_{[Z]}\rangle_{1}-\langle[I_M],\delta_{[Z]}\rangle_{1}\\
&=\delta_{[M],[Z]}+\delta_{[\Omega^{-1}M],[Z]}-\delta_{[I_M],[Z]}.
\end{align*}\end{linenomath}

(2) Suppose $M$ is not projective. For any indecomposable module $X$, we have
\begin{linenomath}\begin{align*}
\langle[X],\delta_{[M]}+\delta_{[\Omega^{-1}M]}\delta_{[P_{\kk}]}\rangle_1
&=\langle[X],\delta_{[M]}\rangle_1+\langle[X]\delta^*_{[P_{\kk}]},\delta_{[\Omega^{-1}M]}\rangle_1\ \text{by\ Proposition}\ \ref{prop2.6} (2)\\
&=\langle[X],\delta_{[M]}\rangle_1+\langle[I_X]-[\Omega^{-1}X],\delta_{[\Omega^{-1}M]}\rangle_1\\
&=\delta_{[X],[M]}-\delta_{[\Omega^{-1}X],[\Omega^{-1}M]}+\delta_{[I_X],[\Omega^{-1}M]}\\
&=0.
\end{align*}\end{linenomath}
Thus, $\delta_{[M]}+\delta_{[\Omega^{-1}M]}\delta_{[P_{\kk}]}=0$ since the form $\langle-,-\rangle_1$ is non-degenerate. This shows that $\delta_{[M]}\delta^*_{[P_{\kk}]}= -\delta_{[\Omega^{-1}M]}$. Now suppose $M$ is projective, for any indecomposable module $X$, we have
\begin{linenomath}\begin{align*}
\langle[X],\delta_{[M]}-\text{Soc}M\delta_{[P_{\kk}]}\rangle_1&=\langle[X],\delta_{[M]}\rangle_1-\langle[X],\text{Soc}M\delta_{[P_{\kk}]}\rangle_1\\
&=\langle[X],\delta_{[M]}\rangle_1-\langle[X],\text{Soc}M\rangle_2\\
&=\delta_{[X],[M]}-\delta_{[X],[M]} \ \ \text{by\ Corollary}\ \ref{cor2015}\\
&=0.
\end{align*}\end{linenomath} Thus, $\delta_{[M]}=\text{Soc}M\delta_{[P_{\kk}]}$, and hence $\delta_{[M]}\delta^*_{[P_{\kk}]}=\text{Soc}M.$
\qed

Recall that an $H$-module $M$ is called \textit{periodic of period $n$} if $\Omega^nM\cong M$ for a minimal natural $n$ (see e.g., \cite{Car}).

\begin{thm}Let $H$ be of finite representation type. The radical of the form $\langle-,-\rangle_3$ is equal to $\mathcal {P}$ if and only if there are no periodic modules of even period.
\end{thm}
\proof Note that the ideal $\mathcal {P}$ of $r(H)$ is contained in the radical of the form $\langle-,-\rangle_3$. If $\mathcal {P}$ is properly contained in the radical of the form $\langle-,-\rangle_3$, there exist some indecomposable non-projective $H$-modules $M$ such that $\sum_{[M]}\lambda_{[M]}[M]$ is a non-zero element in the radical of the form. For any indecomposable non-projective module $Z$, by Lemma \ref{g102} (1), we have
$0=\langle\sum_{[M]}\lambda_{[M]}[M],\delta_{[\Omega^iZ]}\rangle_3=\lambda_{[\Omega^iZ]}+\lambda_{[\Omega^{i+1}Z]}.$
It follows that the following equations hold:
\begin{linenomath}$$\lambda_{[\Omega^iZ]}=(-1)^i\lambda_{[Z]}.$$\end{linenomath}
This implies that $\lambda_{[Z]}=0$ if $Z$ is a periodic module of odd period. However, $\sum_{[M]}\lambda_{[M]}[M]$ is not zero, implying that there exists a periodic module $M$ of even period with $\lambda_{[M]}\neq0$.
Conversely, Suppose the radical of the form $\langle-,-\rangle_3$ is equal to $\mathcal {P}$ and $M$ is a periodic module of even period $2s$. It follows from Lemma \ref{g102} (2) that $\sum_{i=1}^{2s}(-1)^i\delta_{[\Omega^iM]}(1-\delta^*_{[P_{\kk}]})=0$. Thus,  $\sum_{i=1}^{2s}(-1)^i\delta_{[\Omega^iM]}$ belongs to the radical of the form $\langle-,-\rangle_3$. Let $\sum_{i=1}^{2s}(-1)^i\delta_{[\Omega^iM]}=\sum_{j}\lambda_j[P_j]$ for some indecomposable projective modules $P_j$. By Remark \ref{rem2016}, $[P_j]$ can be written as $[P_j]=\sum_{[M]\in\textrm{ind}(H)}\langle [M],[P_j]\rangle_1\delta_{[M]}$. It follows that \begin{linenomath}$$\sum_{i=1}^{2s}(-1)^i\delta_{[\Omega^iM]}=\sum_{j}\sum_{[M]\in\textrm{ind}(H)}\lambda_j\langle [M],[P_j]\rangle_1\delta_{[M]}.$$\end{linenomath}
Comparing the coefficient of $\delta_{[\Omega^i M]}$ in both two sides of the above equality, we obtain that
\begin{linenomath}
$$
(-1)^i=\sum_j\lambda_j\dim_{\kk}\text{Hom}_H(\Omega^i M, P_j)=\sum_j\lambda_j\dim_{\kk}\text{Ext}^i_H(M,P_j)=0,
$$
\end{linenomath}
a contradiction. \qed

\section{\bf Some ring theoretic properties of Green rings}
In this section, we use an associative non-degenerate bilinear form to explore some ring theoretic properties of the Green ring $r(H)$ of $H$. We show that the Green ring $r(H)$ is a Frobenius algebra over $\mathbb{Z}$ if $H$ is of finite representation type. We describe the relation between the Green ring $r(H)$ and the Grothendieck ring $G_0(H)$ of $H$. We give several one-sided ideals of $r(H)$, which are useful to describe the nilpotent radical and central primitive idempotents of $r(H)$.

Note that the $\mathbb{Z}$-bilinear form $\langle-,-\rangle_1$ is not associative in general. However, we may modify it as follows:
\begin{equation}\label{eq3.5}
([X],[Y]):=\langle[X],[Y]^*\rangle_1=\dim_{\kk} \textrm{Hom}_H(X,Y^*).
\end{equation}
Then $(-,-)$ extends to a $\mathbb{Z}$-bilinear form on $r(H)$.

\begin{lem}\label{gg16}
For $X,Y,Z\in H$-mod, the form $(-,-)$ satisfies the following:
\begin{enumerate}
\item[$(1)$] $([X][Y],[Z])=([X],[Y][Z]).$
\item[$(2)$] $([X],[Y])=([Y]^{**},[X]).$ If $S^2$ is inner, then $([X],[Y])=([Y],[X]).$
\end{enumerate}
\end{lem}
\proof (1) The associativity of the form follows from Lemma \ref{glem4} (1), i.e.,
\begin{linenomath}
\begin{align*}
([X][Y],[Z])&=\dim_{\kk}\textrm{Hom}_H(X\otimes Y,Z^*)\\
&=\dim_{\kk}\textrm{Hom}_H(X,(Y\otimes Z)^*)\\
&=([X],[Y][Z]).
\end{align*}
\end{linenomath}

(2) The $\kk$-linear isomorphism $\textrm{Hom}_H(X,Y^*)\cong \textrm{Hom}_H(Y^{**},X^*)$ (Lemma \ref{glem4}, see also \cite{Lo}) implies that $([X],[Y])=([Y]^{**},[X]).$ If $S^2$ is inner, the anti-automorphism $\ast$ of $r(H)$ is an involution. Hence $([X],[Y])=([Y],[X]).$
\qed

The following result can be deduced directly from Lemma \ref{gg15}.

\begin{lem}\label{gg17}
The following hold in $r(H)$:
\begin{enumerate}
\item[$(1)$] For any two indecomposable modules $X$ and $Z$, $(\delta^*_{[Z]},[X])=\delta_{[Z],[X]}$.
\item[$(2)$] For any $x\in r(H)$, $x=\sum_{[M]\in\textrm{ind}(H)}(\delta^*_{[M]},x)[M]$.
\item[$(3)$] The form $(-,-)$ is non-degenerate.
\end{enumerate}
\end{lem}

As an immediate consequence we obtain the following Frobenius property of $r(H)$.

\begin{prop}\label{gg18}
Let $H$ be of finite representation type. The Green ring $r(H)$ is a Frobenius $\mathbb{Z}$-algebra. Moreover, $r(H)$ is a symmetric $\mathbb{Z}$-algebra if the square of the antipode of $H$ is inner.
\end{prop}
\proof Note that $r(H)^{\vee}:=\text{Hom}_{\mathbb{Z}}(r(H),\mathbb{Z})$ is a $(r(H),r(H))$-bimodule via $(afb)(x)=f(bxa)$, for $a,b,x\in r(H)$ and $f\in r(H)^{\vee}$. Since $H$ is of finite representation type, the form $(-,-)$ is associative and non-degenerate with a pair of dual bases $\{\delta^*_{[M]},[M]\mid [M]\in\textrm{ind}(H)\}$. Thus, the map $\rho$ from $r(H)$ to $r(H)^{\vee}$ given by $x\mapsto (-,x)$ is a left $r(H)$-module isomorphism, and hence $r(H)$ is a Frobenius $\mathbb{Z}$-algebra. Moreover, if the square of the antipode is inner, the bilinear form is symmetric and hence $\rho$ is a $(r(H),r(H))$-bimodule isomorphism. It follows that $r(H)$ is a symmetric $\mathbb{Z}$-algebra. \qed

\begin{rem}\label{rem2015} Let $H$ be of finite representation type.
\begin{enumerate}
  \item The Green ring $r(H)$ is a Frobenius $\mathbb{Z}$-algebra with a pair of dual bases $\{\delta^*_{[M]},[M]\mid [M]\in\textrm{ind}(H)\}$ with respect to the form $(-,-)$. The equality of Lemma \ref{gg17} (2) is now equivalent to the following equality:
\begin{linenomath}
\begin{equation*}x=\sum_{[M]\in\textrm{ind}(H)}(x,[M])\delta^*_{[M]},\ \text{for}\ x\in r(H).
\end{equation*}
\end{linenomath}
This means that the transformation matrix from the dual basis $\{\delta^*_{[M]}\mid [M]\in\textrm{ind}(H)\}$  to the standard  basis $\text{ind}(H)$ is an invertible integer matrix with entries $([X],[Y])=\dim_{\kk}\text{Hom}_H(X,Y^*)$ for $[X],[Y]\in\textrm{ind}(H)$.

\item If $H$ is semisimple, then $S^2$ is inner \cite{LR} and $\delta^*_{[M]}=[M]^*=[M^*].$ In this case, $r(H)=G_0(H)$ is symmetric (see Proposition \ref{gg18}) and semiprime \cite{Lo} with a pair of dual bases $\{[M^*],[M]\mid [M]\in\textrm{ind}(H)\}$. We refer to \cite{Wi2} for more details in the semisimple case.
\end{enumerate}
\end{rem}

The bilinear form $(-,-)$ can be used to describe the relation between the Green ring $r(H)$ and the Grothendieck ring $G_0(H)$ of $H$.
Let $\mathcal {P}^{\perp}$ be the subgroup of $r(H)$ which is orthogonal to $\mathcal {P}$ with respect to the form $(-,-)$. Then $\mathcal {P}^{\perp}$ is a two-sided ideal of $r(H)$.

\begin{prop}\label{prop2015} The Grothendieck ring $G_0(H)$ is isomorphic to the quotient ring $r(H)/\mathcal {P}^{\perp}$.
\end{prop}
\proof   Observe that the natural morphism  $\varphi$ given in $(\ref{equgel})$ is surjective. It is sufficient to show that $\ker\varphi=\mathcal {P}^{\perp}$ . Suppose  $\sum_{[M]\in\text{ind}(H)}\lambda_{[M]}[M]\in\ker\varphi$, where $\lambda_{[M]}\in\mathbb{Z}$. Then
\begin{linenomath}
$$\sum_{[V]}\sum_{[M]\in\text{ind}(H)}\lambda_{[M]}[M:V][V]=0.$$
\end{linenomath} Note that a short exact sequence tensoring over $\kk$ with a projective module $P$ is split. It follows that $[M][P]=\sum_{[V]}[M:V][V][P]$ in $r(H)$, and hence
\begin{linenomath}
\begin{align*}
(\sum_{[M]\in\text{ind}(H)}\lambda_{[M]}[M],[P])&=(\sum_{[M]\in\text{ind}(H)}\lambda_{[M]}[M][P],[\kk])\\
&=(\sum_{[V]}\sum_{[M]\in\text{ind}(H)}\lambda_{[M]}[M:V][V][P],[\kk])\\
&=0.
\end{align*}
\end{linenomath}
This implies that $\sum_{[M]\in\text{ind}(H)}\lambda_{[M]}[M]\in\mathcal {P}^{\perp}$.
Now, we assume $\sum_{[M]\in\text{ind}(H)}\lambda_{[M]}[M]\in\mathcal {P}^{\perp}$. Note that $[P]y\in\mathcal {P}$ for any $y\in r(H)$ and $[P]\in \mathcal{P}$. We have
\begin{linenomath}
$$(\sum_{[M]\in\text{ind}(H)}\lambda_{[M]}[M][P],y)=(\sum_{[M]\in\text{ind}(H)}\lambda_{[M]}[M],[P]y)=0.$$
\end{linenomath}
This implies that  $\sum_{[M]\in\text{ind}(H)}\lambda_{[M]}[M][P]=0$  as  the form $(-,-)$ is non-degenerate.  Replacing  $[M][P]$ by $\sum_{[V]}[M:V][V][P]$, we obtain the following  equality:
\begin{equation}\label{2014.12.21}\sum_{[V]}\sum_{[M]\in\text{ind}(H)}\lambda_{[M]}[M:V][V][P]=0.\end{equation} Note that the multiplication $[V][P]=[V\otimes P]$ means that $\mathcal {P}$ is a $G_0(H)$-module. Moreover, the $G_0(H)$-module $\mathcal {P}$ is faithful, see \cite[Section 3.1]{Lo}. It follows from (\ref{2014.12.21}) that $\sum_{[V]}\sum_{[M]\in\text{ind}(H)}\lambda_{[M]}[M:V][V]=0$.
Thus, $\sum_{[M]\in\text{ind}(H)}\lambda_{[M]}[M]\in\ker\varphi$. \qed

Now we turn to the special element $\delta_{[\kk]}$, which plays an important role in the study of the Green ring $r(H)$. For any indecomposable module $X$, the elements $[X]$, $\delta_{[\kk]}$ and $\delta_{[X]}$ satisfy the following relations.

\begin{thm}\label{gg19}Let $X$ be an indecomposable $H$-module.
\begin{enumerate}
  \item[(1)] $\kk\nmid X^*\otimes X$ if and only if $[X]\delta_{[\kk]}=0$.
  \item[(2)] $\kk\nmid X\otimes X^*$ if and only if $\delta_{[\kk]}[X]=0$.
  \item[(3)] $\kk\mid X^*\otimes X$ if and only if $[X]\delta_{[\kk]}=\delta_{[X]}$.
  \item[(4)] $\kk\mid X\otimes X^*$ if and only if $\delta_{[\kk]}[X]=\delta_{[X]}$.
\end{enumerate}
\end{thm}
\proof If $H$ is semisimple, then $\kk\mid X^*\otimes X$ and $\kk\mid X\otimes X^*$. In this case, Part (3) and Part (4) hold obviously because $\delta_{[\kk]}=[\kk]$ and $\delta_{[X]}=[X]$.  Assume  $H$ is not semisimple, we only show Part (1) and Part (3) and the proofs of Part (2) and Part (4) are similar.

(1) If $\kk\nmid X^*\otimes X$, by Proposition \ref{gg10}, the map $id_X\otimes\sigma$ is a split epimorphism. It follows from (\ref{equ4}) that $[X\otimes E]=[X\otimes\tau(\kk)]+[X]$, and hence $[X]\delta_{[\kk]}=0$. Conversely, if  $[X]\delta_{[\kk]}=0$, then $0=(([X]\delta_{[\kk]})^*,[X])=(\delta_{[\kk]}^*,[X]^*[X]).$ This means that $\kk\nmid X^*\otimes X$.

(3) If $\kk\mid X^*\otimes X$, then the map $id_X\otimes\sigma$ is right almost split by Proposition \ref{gg11}. It follows from (\ref{equ3}) and Proposition \ref{gg14} that $\delta_{[X]}=[X\otimes\tau(\kk)]-[X\otimes E]+[X]=[X]\delta_{[\kk]}.$  Conversely, if $[X]\delta_{[\kk]}=\delta_{[X]}$, then $1=(\delta_{[X]}^*,[X])=(([X]\delta_{[\kk]})^*,[X])=(\delta_{[\kk]}^*,[X]^*[X]).$ It follows that $\kk\mid X^*\otimes X$.
\qed

As an application of Theorem \ref{gg19}, we are able to determine the multiplicity of the trivial module $\kk$ in the decomposition of the tensor product $X\otimes X^*$ and $X^*\otimes X$ respectively. For the case where $H$ is semisimple over the field $\kk$ of characteristic 0, this was done by Zhu \cite[Lemma 1]{Zhu}, see also \cite[Proposition 2.1]{Wi2}.
\begin{cor}\label{gg19.1}
Let $X$ be an indecomposable $H$-module.
\begin{enumerate}
  \item[(1)] The multiplicity of $\kk$ in $X^*\otimes X$ is either 0 or 1.
  \item[(2)] The multiplicity of $\kk$ in $X\otimes X^*$ is either 0 or 1.
\end{enumerate}
\end{cor}
\proof (1) We only prove Part (1), the proof of Part (2) is similar.
By Theorem \ref{gg19}, we have
\begin{linenomath}
$$(\delta^*_{[\kk]},[X^*][X])=(([X]\delta_{[\kk]})^*,[X])=\begin{cases}
0, & \kk\nmid X^*\otimes X,\\
1, & \kk\mid X^*\otimes X,
\end{cases}$$
\end{linenomath}
as desired.
\qed

The following result can be deduced from Theorem \ref{gg19}.
\begin{prop}\label{gprop3} Let $0\rightarrow X\rightarrow Y\rightarrow Z\rightarrow0$ be an almost split sequence of $H$-modules.
\begin{enumerate}
 \item[(1)] $\kk\mid Z\otimes Z^*$ if and only if $\kk\mid X\otimes X^*$.
 \item[(2)] $\kk\mid Z^*\otimes Z$ if and only if $\kk\mid X^*\otimes X$.
\end{enumerate}
\end{prop}
\proof We only prove Part (1) because Part (2) follows from Part (1). Applying the duality functor $*$ to the almost split sequence $0\rightarrow X\rightarrow Y\rightarrow Z\rightarrow0$, we get the almost split sequence $0\rightarrow Z^*\rightarrow Y^*\rightarrow X^*\rightarrow0$, see \cite[P.144]{ARS}. Note that both $Z$ and $X^*$ are indecomposable \cite[Proposition 1.14, ChV]{ARS}. This implies that $\delta_{[Z]}^*=\delta_{[X^*]}$. If $\kk\mid Z\otimes Z^*$, by Theorem \ref{gg19}, $\delta_{[\kk]}[Z]=\delta_{[Z]}$. We claim that $\kk\mid X\otimes X^*$. Otherwise, $\kk\nmid X\otimes X^*$, and hence $\kk\nmid X^{**}\otimes X^{*}$. This leads to $[X^*]\delta_{[\kk]}=0$ by Theorem \ref{gg19}. However,
\begin{linenomath}
$$1=(\delta^*_{[X^*]},[X^*])=(\delta^{**}_{[Z]},[X^*])=([X^*],\delta_{[Z]})=([X^*]\delta_{[\kk]},[Z])=0,$$
\end{linenomath} a contradiction. Conversely, if $\kk\mid X\otimes X^*$, then $\kk\mid X^{**}\otimes X^*$. This yields that $[X^*]\delta_{[\kk]}=\delta_{[X^*]}$. We claim that $\kk\mid Z\otimes Z^*$. Otherwise, $\delta_{[\kk]}[Z]=0$ by Theorem \ref{gg19}. Then \begin{linenomath} $$1=(\delta^*_{[Z]},[Z])=(\delta_{[X^*]},[Z])=([X^*],\delta_{[\kk]}[Z])=0,$$\end{linenomath}  a contradiction. \qed

Denote by $\mathcal {J}_{+}$ and $\mathcal {J}_{-}$ the subgroups of $r(H)$ respectively  as follows:
\begin{linenomath} $$\mathcal {J}_{+}:=\mathbb{Z}\{\delta_{[M]}\mid [M]\in\text{ind}(H)\ \text{and}\ \kk\mid M\otimes M^*\},$$\end{linenomath}
\begin{linenomath} $$\mathcal {J}_{-}:=\mathbb{Z}\{\delta_{[M]}\mid [M]\in\text{ind}(H)\ \text{and}\ \kk\mid M^*\otimes M\}.$$\end{linenomath}
By Theorem \ref{gg19},  $\mathcal {J}_{+}$ (resp. $\mathcal {J}_{-}$)  is a right (resp. left)  ideal of $r(H)$ generated by $\delta_{[\kk]}$.  Moreover,  we have   $\mathcal {J}^{*}_{+}=\mathcal {J}_{-}$ and $\mathcal {J}^{*}_{-}=\mathcal {J}_{+}$ by Proposition \ref{gprop3}.

Now let $\mathcal {P}_{+}$ and $\mathcal {P}_{-}$ denote the subgroups of $r(H)$ as follows:
\begin{linenomath} $$\mathcal {P}_{+}:=\mathbb{Z}\{[M]\in\text{ind}(H)\mid \kk\nmid M\otimes M^*\},$$
$$\mathcal {P}_{-}:=\mathbb{Z}\{[M]\in\text{ind}(H)\mid \kk\nmid M^*\otimes M\}.$$\end{linenomath}
It follows from Proposition \ref{gg13} that $\mathcal {P}_{+}$ is a right ideal of $r(H)$ and $\mathcal {P}_{-}$ is a left ideal of $r(H)$.
Obviously, $\mathcal {P}_{+}^*=\mathcal {P}_{-}$ and $\mathcal {P}_{-}^*=\mathcal {P}_{+}$.

According to the associativity and the non-degeneracy of the form $(-,-)$, we have $\mathcal {P}_{-}x=0$ if and only if $(\mathcal {P}_{-},x)=0$ if and only if  $(x,\mathcal {P}_{-})=0$ since $\mathcal {P}_{-}=\mathcal {P}_{-}^{**}$.  Similarly, $x\mathcal {P}_{+}=0$ if and only if $(x,\mathcal {P}_{+})=0$ if and only if $(\mathcal {P}_{+},x)=0$. Thus, the right annihilator $r(\mathcal {P}_{-})$ of $\mathcal {P}_{-}$ and left annihilator $l(\mathcal {P}_{+})$ of $\mathcal {P}_{+}$ can be expressed respectively as follows:
\begin{linenomath} $$r(\mathcal {P}_-):=\{x\in r(H)\mid(x,y)=0\ \textrm{for}\ y\in\mathcal {P}_-\},$$
$$l(\mathcal {P}_+):=\{x\in r(H)\mid(y,x)=0\ \textrm{for}\ y\in\mathcal {P}_+\}.$$\end{linenomath}
The relations between these one-sided ideals of $r(H)$ can be described as follows.

\begin{prop}\label{gg22}
Let $H$ be of finite representation type.
\begin{enumerate}
  \item[(1)] $\mathcal {J}_{+}=r(\mathcal {P}_-)$.
  \item[(2)] $\mathcal {J}_{-}=l(\mathcal {P}_+)$.
\end{enumerate}
\end{prop}
\proof It is sufficient to prove Part (1) and the proof of Part (2) is similar. For any two indecomposable modules $X$ and $Y$ satisfying $\kk\mid X\otimes X^*$ and $\kk\nmid Y^*\otimes Y$, by Theorem \ref{gg19}, we have
\begin{linenomath} $$(\delta_{[X]},[Y])=(\delta_{[\kk]}[X],[Y])=([Y^{**}]\delta_{[\kk]},[X])=(0,[X])=0.$$\end{linenomath}
This implies that $\mathcal {J}_{+}\subseteq r(\mathcal {P}_-)$. For any $x\in r(\mathcal {P}_-),$
\begin{linenomath}
\begin{align*}
x&=\sum_{[M]\in\textrm{ind}(H)}(x,[M])\delta^*_{[M]}\ \ \text{by\ Remark}\ \ref{rem2015} (1)\\
&=\sum_{\kk\mid M^*\otimes M}(x,[M])\delta^*_{[M]}\ \ \text{as}\ x\in r(\mathcal {P}_-).
\end{align*}
\end{linenomath}
We have that $x\in \mathcal {J}_{-}^*=\mathcal {J}_{+}$, and hence $r(\mathcal {P}_-)\subseteq \mathcal {J}_{+}$.
\qed

In the following, we shall use these one-sided ideals to get information about the nilpotent radical and central primitive idempotents of $r(H)$.
We first need the following useful lemma.
\begin{lem}\label{gg19.2} For any $x\in r(H)$, we have the following:
\begin{enumerate}
  \item[(1)] If $xx^*=0$, then $x\in\mathcal {P}_{+}$.
  \item[(2)] If $x^*x=0$, then $x\in\mathcal {P}_{-}$.
\end{enumerate}
\end{lem}
\proof It suffices to prove Part (1), the proof of Part (2) is similar. Suppose
\begin{linenomath} $$x=\sum_{\kk\mid M\otimes M^*}\lambda_{[M]}[M]+
\sum_{\kk\nmid M\otimes M^*}\lambda_{[M]}[M],$$
\end{linenomath} where $\lambda_{[M]}\in\mathbb{Z}$. By Theorem \ref{thm2.7} (1) and Corollary \ref{gg19.1}, the coefficient of the identity $[\kk]$ in the linear expression of $xx^*$ with respect to the basis $\text{ind}(H)$ is $\sum_{\kk\mid M\otimes M^*}\lambda_{[M]}^2$. Thus, if $xx^*=0$, then $\lambda_{[M]}=0$ for any indecomposable module $M$ satisfying  $\kk\mid M\otimes M^*$. Hence $x=\sum_{\kk\nmid M\otimes M^*}\lambda_{[M]}[M]\in\mathcal {P}_{+}.$
\qed

\begin{prop}\label{gg20}
The nilpotent radical of $r(H)$ is contained in $\mathcal {P}_{+}\cap\mathcal {P}_{-}$.
\end{prop}
\proof We assume that the nilpotent radical rad$r(H)$ of $r(H)$ satisfying $(\text{rad}r(H))^m=0$ for some $m>0$. For any $x\in \text{rad}r(H)$, let $x_0:=x$ and $x_{i+1}:=x_ix_i^*$ for $i\geq0$.
Note that $(\text{rad}r(H))^*\subseteq\text{rad}r(H)$. There exists some $k$ such that $x_k=0$. We write
\begin{linenomath} $$x=\sum_{\kk\mid M\otimes M^*}\lambda_{[M]}[M]+
\sum_{\kk\nmid M\otimes M^*}\lambda_{[M]}[M]$$
\end{linenomath}  and
\begin{linenomath} $$x_1=xx^*=\sum_{\kk\mid M\otimes M^*}\mu_{[M]}[M]+
\sum_{\kk\nmid M\otimes M^*}\mu_{[M]}[M],$$
\end{linenomath} for $\lambda_{[M]}$ and $\mu_{[M]}$ in $\mathbb{Z}$. As shown in the proof of Lemma \ref{gg19.2}, the coefficient of $[\kk]$ in $x_1=xx^*$ is $\mu_{[\kk]}=\sum_{\kk\mid M\otimes M^*}\lambda_{[M]}^2$ and the coefficient of $[\kk]$ in $x_2=x_1x_1^*$ is $\sum_{\kk\mid M\otimes M^*}\mu_{[M]}^2$. If $\mu_{[\kk]}\neq0$, then $\sum_{\kk\mid M\otimes M^*}\mu_{[M]}^2\neq0$, and hence $x_2\neq0$. Repeating this process, we obtain that $x_i\neq0$ for any $i\geq 0$. This contradicts to the fact that $x_k=0$. In view of this, $\mu_{[\kk]}=0$, and hence $x=\sum_{\kk\nmid M\otimes M^*}\lambda_{[M]}[M]\in\mathcal {P}_+$. Similarly, if $x\in \text{rad}r(H)$, then $x\in\mathcal {P}_-$. We obtain that $\text{rad}r(H)\subseteq \mathcal {P}_{+}\cap \mathcal {P}_{-}$.
\qed

Now we are able to locate central primitive idempotents of $r(H)$.

\begin{prop}\label{gg21}
Let $e$ be a central primitive idempotent of $r(H)$. Then, either $e\in\mathcal {P}_{+}\cap\mathcal {P}_{-}$ or $1-e\in\mathcal {P}_{+}\cap\mathcal {P}_{-}$.
\end{prop}
\proof If $e$ is a central primitive idempotent of $r(H)$, so is $e^*$ since the duality operator $*$ is an anti-automorphism of $r(H)$. It follows that $e=e^*$ or $ee^*=e^*e=0$. If $ee^*=e^*e=0$, by Lemma \ref{gg19.2}, $e\in\mathcal {P}_+$ and $e\in\mathcal {P}_-$ as well. Now suppose  $e=e^*$, and let
\begin{linenomath} $$e=\sum_{\kk\mid M\otimes M^*}\lambda_{[M]}[M]+
\sum_{\kk\nmid M\otimes M^*}\lambda_{[M]}[M].$$
\end{linenomath}
Comparing the coefficients of $[\kk]$ in both sides of the equation $ee^*=e$, we obtain that $\sum_{\kk\mid M\otimes M^*}\lambda_{[M]}^2=\lambda_{[\kk]}$. This implies that $\lambda_{[\kk]}=0$ or 1, and $\lambda_{[M]}=0$ for any $[M]\neq[\kk]$ satisfying $\kk\mid M\otimes M^*$. Hence $e$ has the following reduced form
\begin{linenomath} $$e=\lambda_{[\kk]}[\kk]+
\sum_{\kk\nmid M\otimes M^*}\lambda_{[M]}[M].$$
In the meanwhile, if we  write
\begin{linenomath} $$e=\sum_{\kk\mid M^*\otimes M}\mu_{[M]}[M]+
\sum_{\kk\nmid M^*\otimes M}\mu_{[M]}[M].$$\end{linenomath}
\end{linenomath}
Then  the equation  $e^*e=e$ yields that
\begin{linenomath} $$e=\mu_{[\kk]}[\kk]+
\sum_{\kk\nmid M^*\otimes M}\mu_{[M]}[M].$$
\end{linenomath}  Thus, $\mu_{[\kk]}=\lambda_{[\kk]}$ which is equal to 0 or 1. We conclude that $e\in\mathcal {P}_{+}\cap\mathcal {P}_{-}$ if $\mu_{[\kk]}=\lambda_{[\kk]}=0$, and $1-e\in\mathcal {P}_{+}\cap\mathcal {P}_{-}$ if $\mu_{[\kk]}=\lambda_{[\kk]}=1$.
\qed

\section{\bf Bi-Frobenius algebra structures on stable Green rings}
In this section, we shall show that the bilinear form $(-,-)$ on the Green ring $r(H)$ could induce a form on the stable Green ring of $H$. The induced form on the stable Green ring is associative, but degenerate in general. We give some equivalent conditions for the non-degeneracy of the form. If the form is non-degenerate, the complexified stable Green algebra is a group-like algebra, and hence a bi-Frobenius algebra.

Recall that the \textit{stable category} $H$-\underline{mod} has the same objects as $H$-mod does, and the space of morphisms from $X$ to $Y$ in $H$-\underline{mod} is the quotient space
\begin{linenomath} $$\text{\underline{Hom}}_H(X,Y):=\text{Hom}_H(X,Y)/\mathcal {P}(X,Y),$$\end{linenomath}
where $\mathcal {P}(X,Y)$ is the subspace of $\text{Hom}_H(X,Y)$ consisting of morphisms factoring through projective modules.
The stable category $H$-\underline{mod} is a triangulated \cite{Hap} monoidal category with the monoidal structure stemming from that of $H$-mod.

\begin{prop}\label{prop20166}
The stable category $H$-\underline{mod} is semisimple if and only if any indecomposable $H$-module is either simple or projective.
\end{prop}
\proof If any indecomposable $H$-module is either simple or projective, using the same method as \cite[Theorem 2.7]{AAI}, one is able to prove that $H$-\underline{mod} is semisimple. Conversely, if $M$ is an indecomposable $H$-module which is neither simple nor projective, then the inclusion map $\text{Soc}M\rightarrow M$ induces the surjective map $M^*\rightarrow (\text{Soc}M)^*$. It follows from Proposition \ref{prop2016} (2) that $\dim_{\kk}\text{\underline{Hom}}_H(M^*,(\text{Soc}M)^*)=\langle M^*,(\text{Soc}M)^*\rangle_3=\langle M^*,(\text{Soc}M)^*\rangle_1\neq0$. This shows that $H$-\underline{mod} is not semisimple.
\qed

The Green ring of the stable category $H$-\underline{mod} is called the \textit{stable Green
ring} of $H$, denoted $r_{st}(H)$. Obviously, the stable Green ring $r_{st}(H)$ admits a $\mathbb{Z}$-basis consisting of all isomorphism classes of indecomposable non-projective $H$-modules.
As the stable category $H$-\underline{mod} is a quotient category of $H$-mod, the stale Green ring $r_{st}(H)$ can be regarded as the quotient ring of the Green ring $r(H)$.

\begin{prop}\label{g100}
The stable Green ring $r_{st}(H)$ is isomorphic to the quotient ring $r(H)/\mathcal{P}$.
\end{prop}
\proof The canonical functor $F$ from $H$-mod to $H$-\underline{mod} given by $F(M)=M$ and $F(\phi)=\underline{\phi}$, for $\phi\in\textrm{Hom}_{H}(M,N)$ with the canonical image $\underline{\phi}\in\textrm{\underline{Hom}}(M,N)$, is a full dense tensor functor. Such a functor induces a ring epimorphism $f$ from $r(H)$ to $r_{st}(H)$ such that $f(\mathcal{P})=0$. Hence there is a unique ring epimorphism $\overline{f}$ from $r(H)/\mathcal{P}$ to $r_{st}(H)$ such that $\overline{f}(\overline{x})=f(x)$, for any $x\in r(H)$ with the canonical image $\overline{x}\in r(H)/\mathcal{P}$. We conclude that $r_{st}(H)$ is isomorphic to $r(H)/\mathcal{P}$ since there is a one to one correspondence between the indecomposable objects in $H$-\underline{mod} and the non-projective indecomposable objects in $H$-mod. \qed


We identify $r(H)/\mathcal{P}$ with $r_{st}(H)$ and denote $\overline{x}$ the element in $r_{st}(H)$ for any $x\in r(H)$.
Since  $(\delta^*_{[\kk]},x)=0$ for any $x\in\mathcal {P}$,  the linear functional $(\delta^*_{[\kk]},-)$ on $r(H)$ induces a linear functional on $r_{st}(H)$. Using this functional, we define a form on $r_{st}(H)$ as follows:
\begin{equation}\label{eq5.1}
[\overline{x},\overline{y}]_{st}:=(\delta^*_{[\kk]},xy),\ \text{for}\ x,y\in r(H).
\end{equation}
 It is obvious that the form $[-,-]_{st}$ is associative and $*$-symmetric: $[\overline{x},\overline{y}]_{st}=[\overline{y}^*,\overline{x}^*]_{st}$.

The \emph{left radical} of the form $[-,-]_{st}$ is the subgroup of $r_{st}(H)$ consisting of $\overline{x}\in r_{st}(H)$ such that $[\overline{x},\overline{y}]_{st}=0$ for any $\overline{y}\in r_{st}(H)$. The \emph{right radical} of the form $[-,-]_{st}$ is defined similarly. The form $[-,-]_{st}$ is non-degenerate if and only if the left radical (or equivalently, the right radical) of the form $[-,-]_{st}$ is zero.

\begin{prop}\label{propgenp} The left radical of the form $[-,-]_{st}$ is equal to $\mathcal {P}_{+}/\mathcal {P}$
and the right radical of the form $[-,-]_{st}$ is equal to $\mathcal {P}_{-}/\mathcal {P}$.
\end{prop}
\proof We only consider the left radical of the form $[-,-]_{st}$. For $x,y\in r(H)$, if $x\in\mathcal {P}_{+}$, then $xy\in\mathcal {P}_{+}$ since $\mathcal {P}_{+}$ is a right ideal of $r(H)$. It follows that $[\overline{x},\overline{y}]_{st}=(\delta^*_{[\kk]},xy)=0$, and hence $\overline{x}$ belongs to the left radical of the form $[-,-]_{st}$.
Conversely, we suppose that $\overline{x}$ belongs to the left radical of the form $[-,-]_{st}$ for $x=\sum_{[M]\in\text{ind}(H)}\lambda_{[M]}[M]$. The inverse of $\ast$ under the composition is denoted  $\star$. For any $[M]\in\text{ind}(H)$, by Theorem \ref{gg19}, we have
\begin{linenomath} $$0=[\overline{x},\overline{[M]^\star}]_{st}=(\delta^*_{[\kk]},x[M]^\star)=([M]^{\star**}\delta^*_{[\kk]},x)=((\delta_{[\kk]}[M])^*,x)
=
\begin{cases}
0, & \kk\nmid M\otimes M^*,\\
\lambda_{[M]}, & \kk\mid M\otimes M^*.
\end{cases}
$$
\end{linenomath}
This implies that $x=\sum_{\kk\nmid M\otimes M^*}\lambda_{[M]}[M]\in\mathcal {P}_{+}$.
\qed

Now let $\mathcal {J}$ be the subgroup of $r(H)$ as follows:
\begin{linenomath} $$\mathcal {J}=\mathbb{Z}\{\delta_{[M]}\mid [M]\in\text{ind}(H)\ \text{and}\ M\ \text{not\ projective}\}.$$\end{linenomath}  Then $\mathcal {J}_{+}$ and $\mathcal {J}_{-}$ are both contained in $\mathcal {J}$. If $H$ is of finite representation type, then $\mathcal {J}$ is nothing but $\ker\varphi$ ($=\mathcal {P}^{\perp})$ by Lemma \ref{gg15} (5). We are now ready to characterize the non-degeneracy of the form $[-,-]_{st}$ using Proposition \ref{propgenp}.

\begin{prop}\label{theogell} The following statements are equivalent:
\begin{enumerate}
  \item The form $[-,-]_{st}$ is non-degenerate.
  \item $\mathcal {P}_{+}=\mathcal {P}_{-}=\mathcal {P}$.
  \item $\mathcal {J}_{+}=\mathcal {J}_{-}=\mathcal {J}$.
  \item $\mathcal {J}$ is an ideal of $r(H)$ generated by the central element $\delta_{[\kk]}$, the left annihilator $l(\mathcal {J})$ and right annihilator $r(\mathcal {J})$ of $\mathcal {J}$ are both equal to $\mathcal {P}$.
\end{enumerate}
\end{prop}
\proof It can be seen from Proposition \ref{propgenp} that Part (1) and Part (2) are equivalent. The equality $\mathcal {P}_{+}=\mathcal {P}$ is equivalent to saying that $\kk\nmid M\otimes M^*$ if and only if $M$ is projective, or equivalently, $\kk\mid M\otimes M^*$ if and only if $M$ is not projective, this is precisely $\mathcal {J}_{+}=\mathcal {J}$. Similarly, $\mathcal {P}_{-}=\mathcal {P}$ if and only if $\mathcal {J}_{-}=\mathcal {J}$.

$(1)\Rightarrow(4)$ If the form $[-,-]_{st}$ is non-degenerate, then $\mathcal {J}_{+}=\mathcal {J}_{-}=\mathcal {J}$. It follows from Theorem \ref{gg19} that $\delta_{[\kk]}$ is a central element of $r(H)$ and $\mathcal {J}$ is an ideal of $r(H)$ generated by $\delta_{[\kk]}$. Observe that $\mathcal {J}_{+}=\mathcal {J}_{-}=\mathcal {J}$ implying that $\mathcal {J}^{*}=\mathcal {J}_{-}^*=\mathcal {J}_{+}=\mathcal {J}$. This deduces that the left and right annihilators of $\mathcal {J}$ coincide: $l(\mathcal {J})=r(\mathcal {J})$. Let $I:=l(\mathcal {J})=r(\mathcal {J})$. We claim that $I=\mathcal {P}$. The inclusion  $\mathcal {P}\subseteq I$ is obvious. We denote $T_{st}=\{\overline{x}\in r_{st}(H)\mid [\overline{x},1]_{st}=0\}$ and $T=\{x\in r(H)\mid \overline{x}\in T_{st}\}$. Then $I\subseteq T$ since $\mathcal {J}x=0$ if and only if $(\mathcal {J},x)=0$. Now $I$ is an ideal of $r(H)$ satisfying $\mathcal {P}\subseteq I\subseteq T$. So $I/\mathcal {P}$ is an ideal of $r_{st}(H)$ contained in $T/\mathcal {P}=T_{st}$. However, $T_{st}$ contains no nonzero ideals of $r_{st}(H)$ since the form $[-,-]_{st}$ is non-degenerate. This implies that $I=\mathcal {P}$.

$(4)\Rightarrow(1)$ If $[\overline{y},\overline{x}]_{st}=0$ for any $y\in r(H)$, then $[\overline{x}^*,\overline{y}^*]_{st}=0$ since the form is $*$-symmetric. We have $0=[\overline{x}^*,\overline{y}^*]_{st}=(\delta^*_{[\kk]},x^*y^*)=((x\delta_{[\kk]})^*,y^*)$. Thus, $x\delta_{[\kk]}=0$, and hence $x\in l(\mathcal {J})=\mathcal {P}$ and $\overline{x}=0$. Similarly, if $[\overline{x},\overline{y}]_{st}=0$ for any $y\in r(H)$, then $\overline{x}=0$.
\qed

\begin{rem}\label{rem201509}
If the form $[-,-]_{st}$ is non-degenerate, then $\mathcal {J}_{+}=\mathcal {J}$. This yields that $\kk\mid M\otimes M^*$ for any indecomposable non-projective module $M$. It deduces that $M\cong M^{**}$ by Theorem \ref{thm2.7} (1). In this case, the operator $*$ on $r_{st}(H)$ is an involution.
\end{rem}

When $H$ is of finite representation type and the form $[-,-]_{st}$ on $r_{st}(H)$ is non-degenerate, we are able to obtain further information about the nilpotent radical of $r(H)$ described as follows.

\begin{thm}\label{theo2015} If $H$ is of finite representation type and the form $[-,-]_{st}$ on $r_{st}(H)$ is non-degenerate, then the nilpotent radical rad$r(H)$ of $r(H)$ is equal to $\mathcal {P}\cap\mathcal {P}^{\perp}$ if and only if $G_0(H)$ is semiprime.
\end{thm}
\proof Since the form $[-,-]_{st}$ on $r_{st}(H)$ is non-degenerate, we have $\mathcal {P}_{+}=\mathcal {P}_{-}=\mathcal {P}$. It follows from Proposition \ref{gg20} that rad$r(H)\subseteq\mathcal {P}$. If $G_0(H)$ is semiprime, the isomorphism $G_0(H)\cong r(H)/\mathcal {P}^{\perp}$ implies that rad$r(H)\subseteq\mathcal {P}^{\perp}$. We obtain that rad$r(H)\subseteq\mathcal {P}\cap\mathcal {P}^{\perp}$. The inclusion $\mathcal {P}\cap\mathcal {P}^{\perp}\subseteq\text{rad}r(H)$ is obvious, since the form $(-,-)$ is non-degenerate. Conversely, if rad$r(H)=\mathcal {P}\cap\mathcal {P}^{\perp}$, it is obvious that $r(H)/\mathcal {P}^{\perp}\cong G_0(H)$ is semiprime.
\qed

\begin{rem}
The map $\varphi:r(H)\rightarrow G_0(H)$ given in (\ref{equgel}) restricting to the ideal $\mathcal {P}$ gives rise to the Cartan map $\varphi|_{\mathcal {P}}:\mathcal {P}\rightarrow G_0(H)$, whose kernel is exactly $\ker(\varphi|_{\mathcal {P}})=\mathcal {P}\cap\ker\varphi=\mathcal {P}\cap\mathcal {P}^{\perp}$.
\end{rem}

\begin{exa}\label{exa201509}
If $H$ is a finite dimensional pointed Hopf algebra of rank one (e.g., Taft algebras and generalized Taft algebras in \cite{CVZ, LZ}, Radford Hopf algebras in \cite{WLZ2}), then $G_0(H)$ is semiprime and the form $[-,-]_{st}$ on $r_{st}(H)$ is non-degenerate since $\mathcal {P}_{+}=\mathcal {P}_{-}=\mathcal {P}$. It follows that rad$r(H)=\mathcal {P}\cap\mathcal {P}^{\perp}=\ker(\varphi|_{\mathcal {P}})$, which is a principal ideal, see \cite{WLZ1, WLZ2} for more details.
\end{exa}

In the rest of this section, we always assume that $H$ is of finite representation type, not semisimple and the form $[-,-]_{st}$ on $r_{st}(H)$ is non-degenerate. In this case, we show that the complexified stable Green algebra $R_{st}(H):=\mathbb{C}\otimes_{\mathbb{Z}}r_{st}(H)$ admits a group-like algebra structure, hence it is a bi-Frobenius algebra, but not a Hopf algebra.

The notion of a group-like algebra was introduced by Doi in \cite{Doi1} generalizing the group algebra of a finite group and a scheme ring (Bose-Mesner algebra) of a non-commutative association scheme.

\begin{defi}\label{defP1} Let $(A,\varepsilon,\mathbf{b},*)$ be a quadruple, where $A$ is a finite dimensional algebra over a field $\kk$ with unit 1, $\varepsilon$ is an algebra morphism from $A$ to $\kk$, the set $\mathbf{b}=\{b_i\mid i\in I\}$ is a $\kk$-basis of $A$ such that $0\in I$ and $b_0=1$, and $*$ is an involution of the index set $I$. Then $(A,\varepsilon,\mathbf{b},*)$ is called a group-like algebra if the following hold:
\begin{enumerate}
  \item[(G1)] $\varepsilon(b_i)=\varepsilon(b_{i^{*}})\neq0$ for all $i\in I$.
  \item[(G2)] $p_{ij}^k=p_{j^*i^*}^{k^*}$ for all $i,j,k\in I$, where $p_{ij}^k$'s are the structure constants for $\mathbf{b}$ defined by $b_ib_j=\sum_{k\in I}p_{ij}^kb_k$.
  \item[(G3)] $p_{ij}^0=\delta_{i,j^*}\varepsilon(b_i)$ for all $i,j\in I$.
\end{enumerate}
\end{defi}

\begin{rem}\label{rem2.6}
\begin{enumerate}
\item Let $(A,\varepsilon,\mathbf{b},*)$ be a group-like algebra. Then $A$ becomes a coalgebra with a comultiplication given by $\bigtriangleup(b_i)=\frac{1}{\varepsilon(b_i)}b_i\otimes b_i$,  see \cite[Remark 3.2]{Doi1}.  Let $\phi\in A^*$ such that $\phi(b_i)=\delta_{0,i}$ and $t=\sum_{i\in I}b_i$. Define the $\kk$-linear map $S$ from $A$ to itself by $S(b_i)=b_{i^*}$ for any $i\in I$. Then $(A,\phi,t,S)$ becomes a bi-Frobenius algebra with a pair of dual bases $\{b_i,\frac{b_{i^*}}{\varepsilon(b_i)}\mid i\in I\}$. We refer to \cite{DT} for the notion of a bi-Frobenius algebra.
\item A group-like algebra is not a Hopf algebra in general. If it is, it must  be a group algebra,  see \cite[Corollary 2]{Ha}.  Thus,  a bi-Frobenius algebra coming from a group-like algebra is not a Hopf algebra if the underlying algebra  is not a group algebra.
\end{enumerate}
\end{rem}

Let $\{[X_i]\mid i\in\mathbb{I}\}$ be the set  of all non-projective indecomposable modules in $\text{ind}(H)$.  By definition $0\in\mathbb{I}$ since $[X_0]:=[\kk]$ is not projective. Note that $X$ is not projective if and only if $X^*$ is not projective. Thus, the duality functor $*$ of $H$-mod induces an involution (see Remark \ref{rem201509}) on the index set $\mathbb{I}$ defined by $[X_{i^*}]:=[X_i^*]$ for any $i\in\mathbb{I}$.

\begin{prop}
The stable Green ring $r_{st}(H)$ is a transitive fusion ring with respect to the basis $\{\overline{[X_i]}\mid i\in\mathbb{I}\}$.
\end{prop}
\proof It is straightforward to verify that $r_{st}(H)$ satisfies the conditions of a fusion ring given in \cite[Definition1.42.2]{etingof}, where the group homomorphism $\tau$ from $r_{st}(H)$ to $\mathbb{Z}$ is determined by $\tau(\overline{x})=(\delta^*_{[\kk]},x)$ for any $\overline{x}\in r_{st}(H)$.
The stable Green ring $r_{st}(H)$ is transitive (\cite[Definition 1.45.1]{etingof}): for any $i,j\in\mathbb{I}$, there exist $k,l\in\mathbb{I}$ such that $\overline{[X_j]}\overline{[X_k]}$ and $\overline{[X_l]}\overline{[X_j]}$ contain $\overline{[X_i]}$ with a nonzero coefficient. In fact, we have $\kk\mid X_j\otimes X_j^*$ since $\mathcal {P}_{+}=\mathcal {P}_{-}=\mathcal {P}$. This implies that $X_i\mid X_j\otimes X_j^*\otimes X_i$. Then we may find an indecomposable non-projective module $X_k$ in $X_j^*\otimes X_i$ such that $X_i\mid X_j\otimes X_k$. Similarly, $X_i\mid X_i\otimes X_j^*\otimes X_j$, then there exists some $X_l$ in $X_i\otimes X_j^*$ such that $X_i\mid X_l\otimes X_j$.
\qed

\begin{rem}
The stable Green ring $r_{st}(H)$ is a fusion ring under the condition that the form $[-,-]_{st}$ on $r_{st}(H)$ is non-degenerate. Nevertheless, the stable category $H$-\underline{mod} is not necessary semisimple by Proposition \ref{prop20166}. A typical example is that the stable category of any Taft algebra of dimension $n^2$ for $n>2$ is not semisimple, while the stable Green ring of the Taft algebra is a fusion ring.
\end{rem}

The fact that $r_{st}(H)$ is a transitive fusion ring enables us to define the Frobenius-Perron dimension of $\overline{[X_i]}$ for any $i\in\mathbb{I}$. Let $\text{FPdim}(\overline{[X_i]})$ be the maximal nonnegative eigenvalue of the matrix of the left multiplication by $\overline{[X_i]}$ with respect to the basis $\{\overline{[X_i]}\mid i\in\mathbb{I}\}$ of $r_{st}(H)$. Then $\text{FPdim}(\overline{[X_i]})$ is called the \textit{Frobenius-Perron dimension} of $\overline{[X_i]}$. Extending $\text{FPdim}$ linearly from the basis $\{\overline{[X_i]}\mid i\in\mathbb{I}\}$ of $r_{st}(H)$ to $R_{st}(H)$, we obtain a functional $\text{FPdim}:R_{st}(H)\rightarrow\mathbb{C}$. The functional $\text{FPdim}$ has the following properties,  see Proposition 1.45.4, Proposition 1.45.5 and Proposition 1.45.8 in \cite{etingof}.

\begin{prop} For any $i\in\mathbb{I}$, we have the following:
\begin{enumerate}
  \item $\text{FPdim}(\overline{[X_i]})\geq1$.
  \item The functional $\text{FPdim}:R_{st}(H)\rightarrow\mathbb{C}$ is a ring homomorphism.
  \item $\text{FPdim}(\overline{[X_i]})=\text{FPdim}(\overline{[X_{i^*}]})$.
\end{enumerate}
\end{prop}

Let $x_i:=\text{FPdim}(\overline{[X_i]})\overline{[X_i]}$ for any $i\in\mathbb{I}$. Then $\mathbf{b}=\{x_i\mid i\in\mathbb{I}\}$ is a basis of $R_{st}(H)$.

\begin{thm}
The quadruple $(R_{st}(H),\text{FPdim},\mathbf{b},*)$ is a group-like algebra.
\end{thm}
\proof We need to verify the conditions (G1)-(G3) given in Definition \ref{defP1}. The condition (G1) is obvious.
To verify the condition (G2), we have
\begin{equation}\label{equ5.4}x_i^*=\text{FPdim}(\overline{[X_i]})(\overline{[X_i]})^*
=\text{FPdim}(\overline{[X_{i^*}]})\overline{[X_{i^*}]}=x_{i^*}.\end{equation}
Now for $i,j\in\mathbb{I}$, we suppose that \begin{equation}\label{equch5.2}x_i x_j=\sum_{k\in\mathbb{I}}p_{ij}^{k}x_k,\end{equation} where  $p_{ij}^k\in\mathbb{C}$. On the one hand, applying the duality operator $*$ to the equality (\ref{equch5.2}) and using (\ref{equ5.4}), we obtain that $x_{j^*} x_{i^*}=\sum_{k\in\mathbb{I}}p_{ij}^{k}x_{k^*}.$ On the other hand, we have  $x_{j^*} x_{i^*}=\sum_{l\in\mathbb{I}}p_{j^*i^*}^{l}x_{l}$. It follows that $p_{ij}^k=p_{j^*i^*}^{k^*}$ for any $i,j,k\in\mathbb{I}$.
Now we verify the condition (G3):
\begin{linenomath}
\begin{align*}
p_{ij}^0&=\text{FPdim}(\overline{[X_i]})\text{FPdim}(\overline{[X_j]})(\delta^*_{[\kk]},[X_i][X_j])\\
&=\text{FPdim}(\overline{[X_i]})\text{FPdim}(\overline{[X_j]})([X_j]^{**}\delta^*_{[\kk]},[X_i])\\
&=\text{FPdim}(\overline{[X_i]})\text{FPdim}(\overline{[X_j]})(\delta^*_{[X_{j^*}]},[X_i])\\
&=\text{FPdim}(\overline{[X_i]})\text{FPdim}(\overline{[X_j]})\delta_{i,j^*}\\
&=\delta_{i,j^*}\text{FPdim}(x_i).
\end{align*}
\end{linenomath}
Therefore, the condition (G3) is satisfied.
\qed

As noted in Remark \ref{rem2.6}, a group-like algebra is a bi-Frobenius algebra. Roughly speaking, a bi-Frobenius algebra $(A,\phi,t,S)$ is both a Frobenius algebra $(A,\phi)$ and a Frobenius coalgebra $(A,t)$ together with an antipode $S$, where $\phi\in \int^r_{A^*}$ and $t\in \int^r_A$  are right integrals of Frobenius algebras \cite{Doi2}.

Now let us look at the bi-Frobenius algebra structure induced from the group-like algebra structure  on $R_{st}(H)$.  The integral  $\phi$ is given by $\phi(x_{i})=\delta_{0,i},$ for $i\in\mathbb{I}$.  Equivalently,
\begin{linenomath} $$\phi(\overline{[X_{i}]})=\begin{cases}
1, & i=0,\\
0, & i\neq0.
\end{cases}$$
\end{linenomath}
The set $\{x_i,\frac{x_i^*}{\text{FPdim}(x_{i})}\mid i\in\mathbb{I}\}$ forms a pair of dual bases of $(R_{st}(H),\phi)$. This is equivalent to saying that $\{\overline{[X_i]},\overline{[X_{i^*}]}\mid i\in\mathbb{I}\}$ is a pair of dual bases of $R_{st}(H)$ with respect to the integral  $\phi$. From the observation above, we conclude that the integral  $\phi$ is nothing but the map determined by the form $[-,-]_{st}$, namely, $\phi(\overline{x})=[\overline{x},1]_{st}$ for $\overline{x}\in R_{st}(H)$.

The stable Green algebra $R_{st}(H)$ is a coalgebra with the counit given by $\text{FPdim}$, and the comultiplication $\bigtriangleup$ defined by $\bigtriangleup(x_{i})=\frac{1}{\text{FPdim}(x_{i})}x_{i}\otimes x_{i},$ or equivalently,
\begin{linenomath} $$\bigtriangleup(\overline{[X_{i}]})=\frac{1}{\text{FPdim}(\overline{[X_{i}]})}\overline{[X_{i}]}\otimes \overline{[X_{i}]},$$\end{linenomath}  for $i\in\mathbb{I}$. Let $t=\sum_{i\in\mathbb{I}}x_{i}=\sum_{i\in\mathbb{I}}\text{FPdim}(\overline{[X_{i}]})\overline{[X_i]}$. Then $t$ is an integral of $R_{st}(H)$ associated to the counit $\text{FPdim}$. Now $(R_{st}(H),t)$ becomes a Frobenius coalgebra. Define a map $S:R_{st}(H)\rightarrow R_{st}(H)$ by $S(x_{i})=x_{i^*},$ that is, $S(\overline{[X_i]})=\overline{[X_{i^*}]}$ for $i\in \mathbb{I}.$ The map $S$ is exactly the duality operator $*$ on $R_{st}(H)$. It is an anti-algebra and anti-coalgebra morphism, so is an antipode of $R_{st}(H)$. Now the quadruple $(R_{st}(H),\phi,t,S)$ forms a bi-Frobenius algebra which is in general not a Hopf algebra.

\section{\bf  Application to Radford Hopf algebras}
In this section, we consider a special finite dimensional pointed Hopf algebra of rank one, known as
a Radford Hopf algebra. We describe the bi-Frobenius algebra structure on the complexified  stable Green algebra of the Hopf algebra from the polynomial point of view.

Given two integers $m>1$ and $n>1$. Let $\omega$ be a primitive $mn$-th root of unity and $H$ an algebra generated by $g$ and $y$ subject to relations
\begin{linenomath} $$g^{mn}=1,\ yg=\omega^{-m}gy,\ y^n=g^n-1.$$\end{linenomath}
Then $H$ is a Hopf algebra whose comultiplication $\bigtriangleup$, counit $\varepsilon$, and
antipode $S$ are given respectively by
\begin{linenomath} $$\bigtriangleup(y)=y\otimes g+1\otimes y,\ \varepsilon(y)=0,\ S(y)=-yg^{-1},$$
$$\bigtriangleup(g)=g\otimes g,\ \varepsilon(g)=1,\ S(g)=g^{-1}.$$\end{linenomath}
The Hopf algebra $H$ is a finite dimensional pointed Hopf algebra of rank one, it was introduced by
Radford in \cite{Rad} so as to give an example of Hopf algebras whose
Jacobson radical is not a Hopf ideal.

The Green ring and the stable Green ring of the Radford Hopf algebra $H$ can be presented by generators and relations.
Let $\mathbb{Z}[Y,Z,X_1,X_2,\cdots,X_{m-1}]$ be a polynomial ring over $\mathbb{Z}$ in  variables $Y,Z,X_1,X_2,\cdots,X_{m-1}$. The Green ring $r(H)$ of $H$ is isomorphic to the quotient ring of $\mathbb{Z}[Y,Z,X_1,X_2,\cdots,X_{m-1}]$ modulo the ideal generated by the  elements from (\ref{equT8}) to (\ref{equT9})  (see \cite[Theorem 8.2]{WLZ2}):
\begin{equation}\label{equT8}Y^n-1,\ (1+Y-Z)F_n(Y,Z),\ YX_1-X_1,\ ZX_1-2X_1,\end{equation}
\begin{equation}\label{equT15}X_1^j-n^{j-1}X_j,\ \textrm{for}\ 1\leq j\leq m-1,\end{equation}
\begin{equation}\label{equT9}X_1^m-n^{m-2}(1+Y+\cdots+Y^{n-1})F_n(Y,Z),\end{equation}
where $F_n(Y,Z)$ is the Dickson polynomial (of the second type) defined recursively by
$F_1(Y,Z)=1,$ $F_2(Y,Z)=Z$ and $F_k(Y,Z)=ZF_{k-1}(Y,Z)-YF_{k-2}(Y,Z)$ for $k\geq3$. More precisely, the polynomial $F_k(y,z)$ can be expressed as follows (see \cite[Lemma 3.11]{CVZ}):
\begin{linenomath} \begin{equation*}F_k(Y,Z)=\sum_{i=0}^{\lfloor\frac{k-1}{2}\rfloor}(-1)^i\binom{k-1-i}{i}Y^iZ^{k-1-2i},\end{equation*}\end{linenomath}
where $\lfloor\frac{k-1}{2}\rfloor$ stands for the biggest integer  not more than $\frac{k-1}{2}$.

The Grothendieck ring $G_0(H)$ of $H$  is isomorphic to the quotient ring of $\mathbb{Z}[Y,X_1,X_2,\cdots,X_{m-1}]$ modulo the ideal generated by $Y^n-1,YX_1-X_1,X_1^j-n^{j-1}X_j$ for $1\leq j\leq m-1$ and $X_1^m-n^{m-1}(1+Y+\cdots +Y^{n-1})$ (see \cite[Corollary 8.3]{WLZ2}).

The stable Green ring $r_{st}(H)$ of $H$ is isomorphic to the stable Green ring of a Taft algebra of dimension $n^2$ (see \cite[Section 7]{WLZ2}), while the latter is isomorphic to the quotient ring $\mathbb{Z}[Y,Z]/I$, where $I$ is the ideal of $\mathbb{Z}[Y,Z]$ generated by $Y^n-1$ and $F_n(Y,Z)$ (see \cite[Proposition 6.1]{WLZ1}).

The form $[-,-]_{st}$ on $r_{st}(H)$ is non-degenerate (see Example \ref{exa201509}). As shown in Section 5, there is a bi-Frobenius algebra structure on the  complexified stable Green algebra $\mathbb{C}\otimes_{\mathbb{Z}}r_{st}(H)\cong\mathbb{C}[Y,Z]/I$. In the following, we shall describe the bi-Frobenius algebra structure on $\mathbb{C}[Y,Z]/I$  using a new basis  rather than the canonical basis  consisting of indecomposable non-projective $H$-modules.
We need the following inverse version of the Dickson polynomials.

\begin{lem}\cite[Lemma 6.4]{WLZ1}\label{lem201509}
For any $j\geq 1$, we have
\begin{linenomath} $$Z^j=\sum_{k=0}^{\lfloor\frac{j}{2}\rfloor}\begin{pmatrix}
                                                           j \\
                                                           k \\
                                                         \end{pmatrix}
\frac{j+1-2k}{j+1-k}Y^kF_{j+1-2k}(Y,Z).$$\end{linenomath}
\end{lem}

Denote by $y^iz^j$ the image of $Y^iZ^j$ under the natural map $\mathbb{C}[Y,Z]\rightarrow\mathbb{C}[Y,Z]/I$. Then the set $\{y^iz^j\mid 0\leq i\leq n-1,0\leq j\leq n-2\}$ forms a basis of $\mathbb{C}[Y,Z]/I$. By Lemma \ref{lem201509}, the following equation holds in $\mathbb{C}[Y,Z]/I$:
\begin{linenomath} $$y^iz^j=\sum_{k=0}^{\lfloor\frac{j}{2}\rfloor}\begin{pmatrix}
                                                           j \\
                                                           k \\
                                                         \end{pmatrix}
\frac{j+1-2k}{j+1-k}y^{i+k}F_{j+1-2k}(y,z).$$
\end{linenomath}  Thus, $\{y^iF_{j}(y,z)\mid 0\leq i\leq n-1,1\leq j\leq n-1\}$ is also a basis of $\mathbb{C}[Y,Z]/I$. In the following, we shall use this basis to describe the bi-Frobenius algebra structure on the algebra  $\mathbb{C}[Y,Z]/I$.  Following from \cite[Remark 4.4 (3)]{WLZ1} we have \begin{equation}\label{equ20155}y^iF_j(y,z)y^kF_l(y,z)=
\sum_{t=\zeta{(j,l)}}^{\min\{j,l\}-1}y^{i+k+t}F_{j+l-1-2t}(y,z),\end{equation}
where $\zeta{(j,l)}=0$ if $j+l-1<n$, and $\zeta{(j,l)}=j+l-n$ if $j+l-1\geq n$.

Define the following two maps $\varepsilon:\mathbb{C}[Y,Z]/I\rightarrow\mathbb{C}$  by
$\varepsilon(y^iF_{j}(y,z))=F_{j}(1,2\cos\frac{\pi}{n})$ and $\bigtriangleup:\mathbb{C}[Y,Z]/I\rightarrow\mathbb{C}[Y,Z]/I\otimes\mathbb{C}[Y,Z]/I$ by $\bigtriangleup(y^iF_j(y,z))=\frac{1}{F_j(1,2\cos\frac{\pi}{n})}y^iF_j(y,z)\otimes y^iF_j(y,z).$ Then both $\varepsilon$ and $\bigtriangleup$ are well-defined since $F_n(1,2\cos\frac{\pi}{n})=0$ (see \cite[Theorem 7.3]{WLZ2}). Moreover, it is straightforward to check that $(\triangle\otimes id)\bigtriangleup=(id \otimes\bigtriangleup)\bigtriangleup$ and $(id \otimes\varepsilon)\bigtriangleup=id=(\varepsilon\otimes id)\bigtriangleup$.  Hence $(\mathbb{C}[Y,Z]/I, \bigtriangleup, \varepsilon)$ is a coalgebra.

Define the linear map $\phi:\mathbb{C}[Y,Z]/I\rightarrow\mathbb{C}$ by
\begin{linenomath} $$\phi(y^iF_j(y,z))=\begin{cases}
1, & i=0,j=1,\\
0, & \text{otherwise}.
\end{cases}$$
\end{linenomath}
Then $(\mathbb{C}[Y,Z]/I,\phi)$ is a Frobenius algebra and
\begin{linenomath} $$\{y^iF_j(y,z),y^{1-i-j}F_j(y,z)\mid 0\leq i\leq n-1,1\leq j\leq n-1\}$$\end{linenomath}  forms a pair of dual bases of $\mathbb{C}[Y,Z]/I$ with respect to the integral $\phi$.

Denote by $t=\sum_{i=0}^{n-1}\sum_{j=1}^{n-1}F_j(1,2\cos\frac{\pi}{n})y^iF_j(y,z)$. Then
\begin{linenomath} $$\bigtriangleup(t)=\sum t_1\otimes t_2=\sum_{i=0}^{n-1}\sum_{j=1}^{n-1}y^iF_j(y,z)\otimes y^iF_j(y,z).$$\end{linenomath}
Define the linear map $S:\mathbb{C}[Y,Z]/I\rightarrow\mathbb{C}[Y,Z]/I$ by
\begin{linenomath} $$S(f)=\sum\phi(t_1f)t_2=\sum_{i=0}^{n-1}\sum_{j=1}^{n-1}\phi(y^iF_j(y,z)f)y^iF_j(y,z).$$\end{linenomath}
We have the following result.

\begin{thm}
The  quadruple $(\mathbb{C}[Y,Z]/I,\phi,t,S)$ is a bi-Frobenius algebra.
\end{thm}
\begin{proof}
According to \cite[Lemma 1.2]{Doi3}, if the map $S$ is both anti-algebra and anti-coalgebra automorphisms, then $(\mathbb{C}[Y,Z]/I,\phi,t,S)$ is a bi-Frobenius algebra. Indeed, by the definition of $S$ and the equality (\ref{equ20155}), we have
\begin{linenomath}
\begin{align*}
S(y^kF_l(y,z))&=\sum_{i=0}^{n-1}\sum_{j=1}^{n-1}\phi(y^iF_j(y,z)y^kF_l(y,z))y^iF_j(y,z)\\
&=\sum_{i=0}^{n-1}\sum_{j=1}^{n-1}\phi(\sum_{t=\zeta{(j,l)}}^{\min\{j,l\}-1}y^{i+k+t}F_{j+l-1-2t}(y,z))y^iF_j(y,z).
\end{align*}
\end{linenomath}
By the definition of $\phi$, we have $\phi(y^{i+k+t}F_{j+l-1-2t}(y,z))=1$ if and only if $i,j$ and $t$ satisfy $n\mid i+k+t$ and $j+l-1-2t=1$. Note that $t\leq\min\{j,l\}-1$. The equality $j+l-1-2t=1$ implies that $j=l$. In this case, $t=l-1$ and $n\mid i+k+l-1$. It follows that $S(y^kF_l(y,z))=y^{1-k-l}F_l(y,z)$ and $S$ maps the basis to its dual basis, hence the map $S$ is bijective. In particular, $S(1)=1$ and
\begin{linenomath}
\begin{align*}
S(y^iF_j(y,z)y^kF_l(y,z))&=S(\sum_{t=\zeta(j,l)}^{\min\{j,l\}-1}y^{i+k+t}F_{j+l-1-2t}(y,z))\\
&=\sum_{t=\zeta(j,l)}^{\min\{j,l\}-1}y^{1-(i+k+t)-(j+l-1-2t)}F_{j+l-1-2t}(y,z)\\
&=y^{1-i-j}F_j(y,z)y^{1-k-l}F_l(y,z)\\
&=S(y^iF_j(y,z))S(y^kF_l(y,z)).
\end{align*}
\end{linenomath}
We conclude that $S$ is an anti-algebra map since $\mathbb{C}[Y,Z]/I$ is a commutative algebra. In addition, \begin{linenomath} $$(\varepsilon\circ S)(y^iF_j(y,z))=\varepsilon(y^{1-i-j}F_j(y,z))=\varepsilon(y^iF_j(y,z))$$\end{linenomath}
and
\begin{linenomath}
\begin{align*}(\bigtriangleup\circ S)(y^{i}F_j(y,z))&=\bigtriangleup(y^{1-i-j}F_j(y,z))\\
&=\frac{1}{F_j(1,2\cos\frac{\pi}{n})}y^{1-i-j}F_j(y,z)\otimes y^{1-i-j}F_j(y,z)\\
&=((S\otimes S)\circ\bigtriangleup^{\text{op}})(y^iF_j(y,z)).
\end{align*}
\end{linenomath}
It follows that $S$ is an anti-coalgebra map on $\mathbb{C}[Y,Z]/I$.
\end{proof}

\begin{rem}
Note that $\{y^iz^j\mid 0\leq i\leq n-1,0\leq j\leq n-2\}$ is a basis of $\mathbb{C}[Y,Z]/I$. The bi-Frobenius algebra structure on $(\mathbb{C}[Y,Z]/I,\phi,t,S)$ can also be described by this basis as follows:
\begin{itemize}
\item $\bigtriangleup(y^iz^j)=\sum_{k=0}^{\lfloor\frac{j}{2}\rfloor}\begin{pmatrix}
                                                           j \\
                                                           k \\
                                                         \end{pmatrix}
\frac{j+1-2k}{(j+1-k)F_{j+1-2k}(1,2\cos\frac{\pi}{n})}y^{i+k}F_{j+1-2k}(y,z)\otimes y^{i+k}F_{j+1-2k}(y,z);$
\item $\phi(y^iz^j)=\begin{cases}
\begin{pmatrix}
                                                           j \\
                                                           \frac{j}{2} \\
                                                         \end{pmatrix}
\frac{2}{j+2}, & 2\mid j\ \text{and}\ n\mid i+\frac{j}{2},\\
0, & \text{otherwise};
\end{cases}$

\item $t=\sum_{i=0}^{n-1}\sum_{j=1}^{n-1}F_j(1,2\cos\frac{\pi}{n})y^iF_j(y,z)$;

\item $S(y^iz^j)=\sum_{k=0}^{\lfloor\frac{j}{2}\rfloor}\begin{pmatrix}
                                                           j \\
                                                           k \\
                                                         \end{pmatrix}
\frac{j+1-2k}{j+1-k}y^{k-i-j}F_{j+1-2k}(y,z).$
\end{itemize}
\end{rem}


\vskip5pt
\begin{thebibliography}{99}
\bibliographystyle{siam}

\bibitem{AAI}
N. Andruskiewitsch, I. Angiono, A. G. Iglesias, et al, From Hopf algebras to tensor categories, Conformal field theories and tensor categories, Springer Berlin Heidelberg, (2014): 1-31.

\bibitem{ARS}
M. Auslander, I. Reiten, S. O. Smal$\phi$, Representation theory of Artin algebras, Cambridge Studies in Advanced Mathematics, Vol.{\bf 36}, Cambridge, 1994.

\bibitem{BA}
B. Bakalov, A. A. Kirillov, Lectures on tensor categories and modular functors, Providence: AMS, 2001.

\bibitem{Be}
D. J. Benson, R. A. Parker, The Green ring of a finite group, J. Algebra {\bf 87} (1984): 290-331.

\bibitem{BC}
D. J. Benson, J. F. Carlson, Nilpotent elements in the Green ring, J. Algebra {\bf104} (1985): 329-350.

\bibitem{Car}
J. F. Carlson, The dimensions of periodic modules over modular group algebras, Illinois J. Math. {\bf 23} (2) (1979): 295-306.

\bibitem {Ch}
H. Chen, The Green ring of Drinfeld double $D(H_4)$, Algebras and Representation Theory {\bf 17} (5) (2014): 1457-1483.

\bibitem{CVZ}
H. Chen, F. V. Oystaeyen and Y. Zhang, The Green  rings of Taft algebras, Proc. Amer. Math. Soc. {\bf 142} (2014): 765-775.

\bibitem{Ci}
C. Cibils, A quiver quantum group, Commun. Math. Phys. {\bf 157}
(1993): 459-477.

\bibitem{DH}
E. Darp\"{o}, M. Herschend, On the representation ring of the polynomial algebra over perfect field, Math. Z {\bf 265} (2011): 605-615.

\bibitem{Doi1}
Y. Doi, Bi-Frobenius algebras and group-like algebras, Lecture notes in pure and applied Mathematics (2004): 143-156.

\bibitem{Doi2}
Y. Doi, Group-like algebras and their representations, Commun. Algebra {\bf 38} (7), (2010): 2635-2655.

\bibitem{Doi3}
Y. Doi, Substructures of bi-Frobenius algebras, J. Algebra {\bf 256} (2002): 568-582.

\bibitem{DT}
Y. Doi, M. Takeuchi, BiFrobenius algebras, Contemp. Math. {\bf 267} (2000): 67-98.

\bibitem{EGST}
K. Erdmann, E. L. Green, N. Snashall, R. Taillefer, Representation theory of the Drinfeld doubles of a family of Hopf algebras, J. Pure Appl. Algebra {\bf 204} (2) (2006): 413-454.

\bibitem{etingof}
P. Etingof, S. Gelaki, D. Nikshych, V. Ostrik, Tensor categories, lecture note for the MIT course 18.769, (2009). available at: www-math.mit.edu/ etingof/tenscat. pdf.

\bibitem{GMS}
E. L. Green, E. N. Marcos, and ${\O}$. Solberg, Representations and almost split sequences for Hopf algebras, Representation theory of algebras (Cocoyoc, 1994) (1996): 237-245.

\bibitem{Ha}
Mariana Haim, Group-like algebras and Hadamard matrices, J. Algebra {\bf 308} (2007): 215-235.

\bibitem{Hap}
D. Happel, Triangulated categories in the representation of finite dimensional algebras, Cambridge University Press, 1988.

\bibitem{HOYZ}
H. Huang, F. V. Oystaeyen, Y. Yang, and Y. Zhang, The Green rings of pointed tensor categories of finite type, J. Pure Appl. Algebra {\bf 218} (2014): 333-342.

\bibitem{LR}
R. G. Larson, D. E. Radford, Semisimple cosemisimple Hopf algebras, Amer. J. Math. {\bf 109} (1987): 187-195.

\bibitem{LH}
Y. Li, N. Hu, The Green rings of the 2-rank Taft algebra and its two relatives twisted, J. Algebra {\bf 410} (2014): 1-35.

\bibitem{LZ}
L. Li, Y. Zhang, The Green rings of the generalized Taft Hopf algebras, Contemp. Math. {\bf 585} (2013): 275-288.

\bibitem{Lo}
M. Lorenz, Representations of finite-dimensional Hopf algebras, J. Algebra {\bf 188} (1997): 476-505.

\bibitem{Mon}
S. Montgomery, Hopf Algebras and their actions on rings, CBMS Series
in Math., Vol. {\bf 82}, AMS, Providence, 1993.

\bibitem{NR}
W.D. Nichols and M. B. Richmond,  The Grothendieck algebra of a Hopf algebra I, Comm. Alg. 26 (1998), no. 4, 1081-1095.

\bibitem{Rad}
D. E. Radford, On the coradical of a finite-dimensional Hopf algebra, Proc. Amer. Math. Soc. {\bf 53} (1) (1975): 9-15.

\bibitem{Swe}
M. E. Sweedler, Hopf Algebras, Benjamin, New York, 1969.

\bibitem{Wa}
M. Wakui, Various structures associated to the representation categories of eight-dimensional nonsemisimple Hopf algebras,
Algebras and Representation Theory {\bf 7} (2004): 491-515.

\bibitem{WLZ1}
Z. Wang, L. Li, Y. Zhang, Green rings of pointed rank one Hopf algebras of nilpotent type, Algebras and Representation Theory {\bf 17} (6) (2014): 1901-1924.

\bibitem{WLZ2}
Z. Wang, L. Li and Y. Zhang, Green rings of pointed rank one Hopf algebras of non-nilpotent type, J. Algebra {\bf 449} (2016): 108-137.

\bibitem{Wi2}
S. J. Witherspoon, The representation ring and the centre of a Hopf algebra, Canad. J. Math. {\bf 51} (4) (1999): 881-896.

\bibitem{Y}
S. Yang, Finite dimensional representations of $u$-Hopf algebras, Commun. Algebra {\bf 29} (12) (2001): 5359-5370.

\bibitem{Zhu}
Y. Zhu, Hopf algebras of prime dimension, Internat. Math. Res. Notices {\bf 1} (1994): 53-59.

\end{thebibliography}
\end{document}